\documentclass{ws-ijbc}
\usepackage{ws-rotating}     
\usepackage{graphicx}
\usepackage{epstopdf}
\usepackage{multicol}

\usepackage{dsfont}
\usepackage{amssymb}
\usepackage{amsmath}
\usepackage{amsfonts}
\usepackage{mathtools}
\usepackage{relsize}

\usepackage{tikz}
\usetikzlibrary{matrix}
\usetikzlibrary{calc,decorations.pathreplacing,fit}
\tikzset{%
  highlight1/.style={rectangle,rounded corners,color=red!,fill=blue!15,draw,fill opacity=0.5,thick,inner sep=0pt}
}
\tikzset{%
  highlight2/.style={rectangle,rounded corners,color=red!,draw,fill opacity=0.5,thick,inner sep=0pt}
}
\tikzset{%
  highlight3/.style={rectangle,rounded corners,color=red!,draw,fill opacity=0.5,thick,inner sep=0pt}
}
\tikzset{
    cheating dash/.code args={on #1 off #2}{
        \csname tikz@addoption\endcsname{%
            \pgfgetpath\currentpath%
            \pgfprocessround{\currentpath}{\currentpath}%
            \csname pgf@decorate@parsesoftpath\endcsname{\currentpath}{\currentpath}%
            \pgfmathparse{\csname pgf@decorate@totalpathlength\endcsname-#1}\let\rest=\pgfmathresult%
            \pgfmathparse{#1+#2}\let\onoff=\pgfmathresult%
            \pgfmathparse{max(floor(\rest/\onoff), 1)}\let\nfullonoff=\pgfmathresult%
            \pgfmathparse{max((\rest-\onoff*\nfullonoff)/\nfullonoff+#2, #2)}\let\offexpand=\pgfmathresult%
            \pgfsetdash{{#1}{\offexpand}}{0pt}}%
    }
}

\begin{document}

\catchline{}{}{}{}{} 

\markboth{V.~N.~Belykh {\it et al.}}{Twisted homoclinic orbits in Lorenz and Chen systems}

\title{Twisted homoclinic orbits in Lorenz and Chen systems:\\ rigorous proofs from universal normal form}

\author{Vladimir~N.~Belykh, Nikita~V.~Barabash \\ and  Anastasia~E.~Suroegina}
\address{Volga State University of Water Transport,\\
	5, Nesterov Str., 603950, Nizhny Novgorod, Russia \\
        Lobachevsky State University of Nizhny Novgorod, \\
        23, Gagarin Ave., 603022, Nizhny Novgorod, Russia \\
	belykh@unn.ru\\ barabash@itmm.unn.ru\\ suroegina@unn.ru}

\maketitle

\begin{history}
\received{(to be inserted by publisher)}
\end{history}

\begin{abstract}
The properties common to the Lorenz and Chen attractors, as well as their fundamental differences, have been studied for many years in a vast number of works and remain a topic far from a rigorous and complete description.
In this paper we take a step towards solving this problem by carrying out a rigorous study of the so-called universal normal form to which we have reduced the systems of both of these families.
For this normal form, we prove the existence of infinite set of homoclinic orbits with different topological structure defined by the number of rotations around axis of symmetry. 
We show that these rotational topological features are inherited by the attractors of Chen-type systems and give rise to their twisted nature -- the generic difference from attractors of Lorenz type.
\end{abstract}


 \begin{multicols}{2}
\section{Introduction
\label{Sec:Introduction}}
 In 1978 there appeared the following 3D system 
 \begin{equation}
     \begin{array}{l}
          \dot{x}=y,\\
          \dot{y}=-\left(x^2+z-1\right)x-\lambda y,\\
          \dot{z}=-\alpha z+\beta x^2,
     \end{array}
     \label{eq: main system}
 \end{equation}
which was obtained from the famous Lorenz system by a one-to-one variables change \cite{Yudovich1978,shimizu1978chaos}.
Here the positive parameters $\alpha,\beta,\lambda$ are obtained  from the Lorenz parameters using the losing orientation mapping so that only a part of these parameters corresponds to the positive parameters of the Lorenz system. 
This was the reason that the authors of \cite{guckenheimer2013nonlinear} called the homoclinic orbits existence theorem proved in \cite{belykh1980qualitative,belykh1984bifurcation} as ``\textit{continuation of homoclinic bifurcations}''. 
\footnote{Some historical details on this theorem proof one can read in \cite{lozi2024paths}.} 

In the paper \cite{shil1993normal} a detailed study of the bifurcations of the system \eqref{eq: main system} was carried out, where several systems with symmetry were reduced to it.
The existence of orientable strange attractor of the system~\eqref{eq: main system} for small $\lambda,\beta$ and $|\alpha-1|$ was proved in the paper \cite{ovsyannikov2016analytic}.
The starting point of the proof was codimension 2 bifurcation corresponding to the existence of two simple symmetric homoclinic orbits to a neutral saddle at the origin.
Below we show that the Chen system \cite{chen1999yet} as well as a set of Lorenz-like systems are the particular cases of the system~\eqref{eq: main system}.

The above allows us to call the system~\eqref{eq: main system} \textit{the universal normal form} (UNF) of generalized Lorenz system.

In this paper, the use of the auxiliary systems method allowed us to significantly expand the analyzed parameter range and demonstrate the possibility of forming more complex orientable and non-orientable attractors.
A key role in their formation belongs to homoclinic orbits of the saddle, twisted around the axis of symmetry.
We call such orbits \textit{twisted homoclinic orbits}.
Such twists can be associated with a third symbol, complementing the well-known Bernoulli scheme of two symbols for classical Lorenz attractors \cite{afraimovich1977origin,afraimovich1982attractive}.

We prove the existence of infinite set of twisted homoclinic orbits having different number of rotations around an invariant axis of symmetry.
We reduce the Lorenz and Chen systems to UNF~\eqref{eq: main system} and show that they belong to two non-overlapping domains of parameters having the common boundary.
We call these two domains Lorenz-like and Chen-like systems and analyze their attractors. 
Finally we present the computer simulation and show that in a small neighborhood of the parameters of the original Chen attractor there are twisted homoclinic orbits that play a major role in its formation, unlike the Lorenz attractor.

\section{Lorenz-like and Chen-like systems \label{Sec:Generalized Lorenz system}}
\noindent The generalized Lorenz system of the form 
\begin{equation}
    \begin{array}{l}
        \dot{x}=-a(x-y),\\
         \dot{y}=(r-z)x-qy,\\
         \dot{z}=xy-bz,
    \end{array}
    \label{eq: generalized Lorenz}
\end{equation}
was proposed in the papers \cite{vcelikovsky2002generalized,leonov2015differences} to unite the famous Lorenz, Chen, Lu and Tigan systems. 
The parameters of these systems are as follows: $a>0$, $b>0$ and 
\begin{equation}
    \begin{array}{lll}
\textrm{Lorenz system:} &\quad  r>0, &q=1,\\ 
\textrm{Chen system:}  &\quad  r=c-a, &q=-c,\, c>0,\\  
\textrm{Lu system:}  &\quad r=0, &q=-c,\, c>0,\\
\textrm{Tigan system:}  &\quad  r=c-a, &q=0.
    \end{array}
    \label{eq:type_of_system}
\end{equation}

\begin{remark}
The canonical form in the paper \cite{wang2013gallery} [see system (6) therein] is the generalized Lorenz system~\eqref{eq: generalized Lorenz} in the case
\begin{equation*}
    \begin{pmatrix}
    a_{11} & a_{12} \\
    a_{21} & a_{22}
    \end{pmatrix}
    =
    \begin{pmatrix}
    -a & a \\
    r & -q
    \end{pmatrix},
    \quad a_{33}=-b,\quad c = 0. 
\end{equation*}
\end{remark}

\begin{proposition}
The one-to-one map 
\begin{equation*}
\begin{array}{c}
      V:\quad (x,y,z,t)  \to \\
     \left[\dfrac{\omega }{\sqrt{2}}x, \ \dfrac{\omega^2a}{\sqrt{2}}(y-x), \ \omega^2\left(az-\dfrac{x^2}{2}\right),\ \dfrac{t}{\omega}\right]
\end{array}
\end{equation*}
where 
\begin{equation*}
 \omega^{-2}=a(r-q)>0,
\end{equation*}
 transforms the generalized Lorenz system \eqref{eq: generalized Lorenz} into UNF \eqref{eq: main system}.
 The parameters mapping for $q=const$
 \begin{equation*}
     P:\quad(a,b,c)\to (\lambda,\alpha,\beta)
 \end{equation*}
 reads
 \begin{equation}
         \lambda=(q+a)\omega,\quad \alpha=b\omega,\quad \beta=(2a-b)\omega,
         \label{eq:lambda_alpha_beta}
 \end{equation}
 where $r>q>-a$.
\label{prop:(1)The one-to-one map}
\end{proposition}
\begin{proof}
    A direct substitution of the inverse map $V^{-1}$ coordinates and time
    \begin{equation*}
    \begin{array}{cc}
        V^{-1}:\quad(x,y,z,t)\to \\
        \left[\dfrac{\sqrt{2}}{\omega}x,\dfrac{\sqrt{2}}{\omega}\left(x+\dfrac{1}{a\omega}y\right),\dfrac{1}{a\omega^2}\left(z+x^2\right), \omega t\right]
    \end{array}
    \end{equation*}
    into the generalized Lorenz system~\eqref{eq: generalized Lorenz} proves the proposition. 
\end{proof}

 The Jacobian of the map $P$ has the form $JP=a\omega^5 q$. Hence, the map $P$ is singular and changes the orientation at $q=0$.
 From~\eqref{eq:lambda_alpha_beta} we obtain the system of equations 
 \begin{equation*}
    \lambda^2a(r-q)=(q+a)^2, \quad
      A^2(r-q)=a,
 \end{equation*}
 where
 \begin{equation}
     A=\dfrac{\alpha+\beta}{2}
 \label{eq:A}
 \end{equation}
is the characteristic parameter such that the map $P$ changes orientation for $\lambda=A$ at $q=0$.
Dividing the first equation by the second we get  the final form of equations characterizing the map $P$
 \begin{equation}
     \begin{array}{l}
         \lambda=A\left(1+\dfrac{q}{a}\right),\quad A^2=\dfrac{a}{r-q}.
     \end{array}
     \label{eq: inverse map P}
 \end{equation}

Using ~\eqref{eq:type_of_system}, \eqref{eq: inverse map P} we present the main statement concerning the relation between UNF~\eqref{eq: main system} and a set of Lorenz-like and Chen-like systems.

\begin{proposition}
\begin{enumerate}
    \item The parameter region of UNF~\eqref{eq: main system}
    \begin{equation}
        \lambda>A
        \label{eq: proposition 2.1}
    \end{equation}
    corresponds to the Lorenz-like systems for $q>0$ (white region in Fig.~\ref{fig:UNF_systems_diag}).
    \item The parameter region of UNF~\eqref{eq: main system}
    \begin{equation}
        \lambda<A
        \label{eq: proposition 2.2}
    \end{equation}
    corresponds to Chen-like systems (gray-blue region in Fig.~\ref{fig:UNF_systems_diag}), such that for  
    \begin{equation}
        \lambda=\frac{A^2-1}{2A}<A
        \label{eq:Chen_like_sys}
    \end{equation} 
    UNF~\eqref{eq: main system} corresponds to the original Chen system.
    \item The parameter region of UNF~\eqref{eq: main system}
    \begin{equation}
        \lambda=\frac{A^2-1}{A}<A,
        \label{eq:Lu}
    \end{equation} 
    corresponds to Lu system.
    \item The dividing line
    \begin{equation}
        \lambda=A
        \label{eq:Tigan}    
    \end{equation} 
    corresponds to Tigan system.
\end{enumerate}
\label{prop:(2)five conditions}
\end{proposition}

\begin{proof}
From \eqref{eq:type_of_system} and \eqref{eq: inverse map P} it follows that for original Lorenz system \eqref{eq: generalized Lorenz}, $q=1$, the parameter $\lambda$ satisfies inequality~\eqref{eq: proposition 2.1}.
The generalized Lorenz system \eqref{eq: generalized Lorenz} satisfies the same condition~\eqref{eq: proposition 2.1} for any $q>0$ only.

UNF~\eqref{eq: main system} is no longer the Lorenz-like system in the region \eqref{eq: proposition 2.2}, $q<0$.

For original Chen system~\eqref{eq: generalized Lorenz}, $r=c-a$, $q=-c<0$, the equation \eqref{eq: inverse map P} takes the from
    \begin{equation*}
        \lambda=\left(1-\frac{c}{a}\right)A, \quad 2\frac{c}{a}-1=\frac{1}{A^2},
    \end{equation*}
    which gives the expression~\eqref{eq:Chen_like_sys}.
    
Similarly for Lu system~\eqref{eq: generalized Lorenz}, $r=0$, $q=-c<0$, instead of~\eqref{eq: inverse map P} we obtain 
    \begin{equation*}
         \lambda=\left(1-\frac{c}{a}\right)A, \ \ \ \frac{c}{a}=\frac{1}{A^2}       
    \end{equation*}
    which gives formula \eqref{eq:Lu}.

Due to~\eqref{eq: inverse map P} similarly to Chen system the positive parameter $\lambda$ satisfies the condition~\eqref{eq: proposition 2.2} for all negative values $q\in(-a, 0)$.
Thus any UNF~\eqref{eq: main system} is Chen-like system in the region \eqref{eq: proposition 2.2}.

The line~\eqref{eq:Tigan} corresponds to Tigan system~\eqref{eq: generalized Lorenz}, $q=0$. 
At this line the map $P$ becomes degenerate and therefore serves the boundary of Lorenz-like and Chen-like regions.
\end{proof}

\begin{figure*}
    \centering   \includegraphics[width=0.4\linewidth]{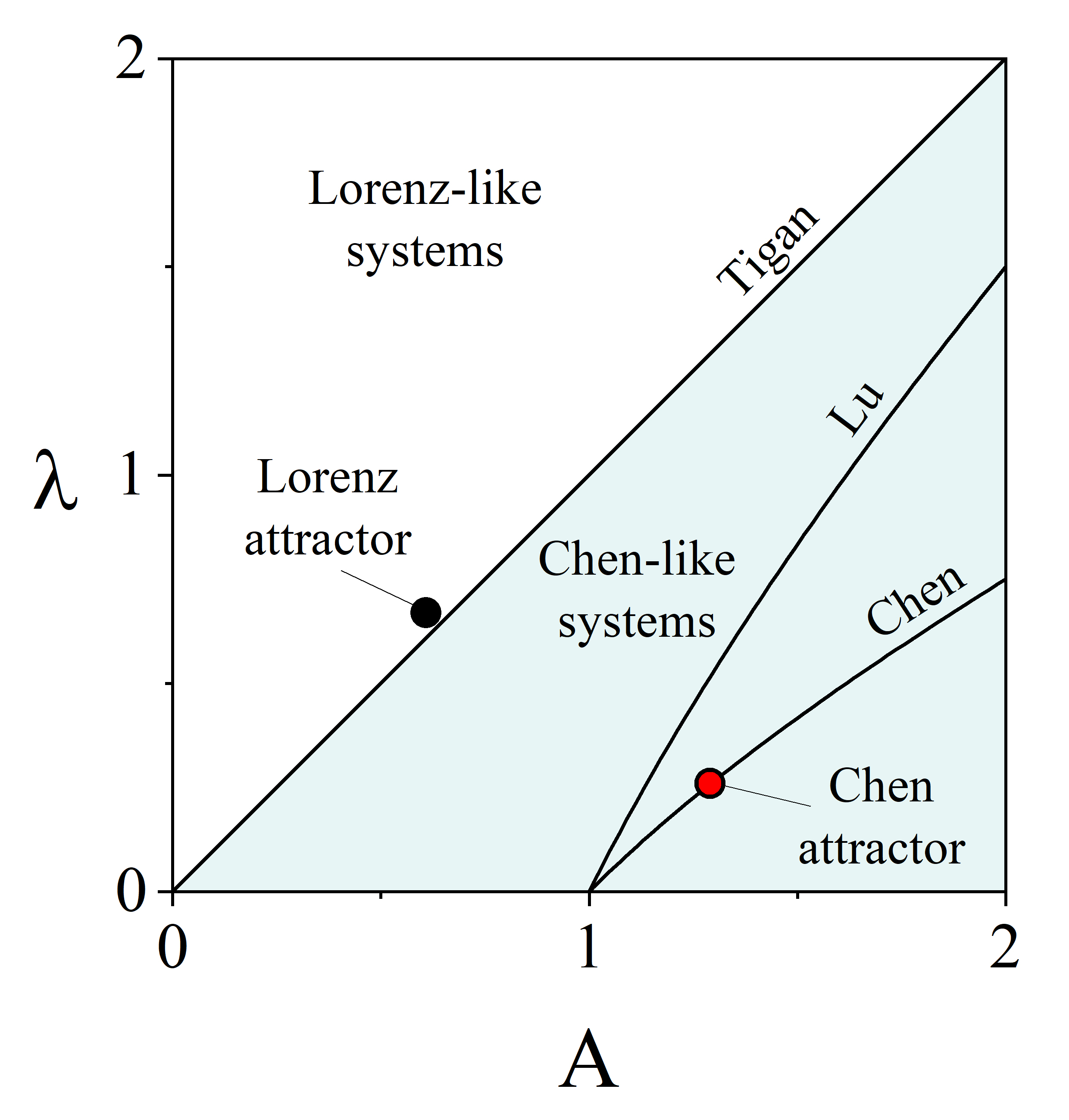}
    \smallskip
    \caption{Partition of the parameter plane of UNF~\eqref{eq: main system} into regions of Lorenz- (white region) and Chen-like (gray-blue region) systems. Parameter $A$ is defined by Eq.~\eqref{eq:A}.
    ``Lu'', ``Chen'' and ``Tigan'' lines are defined by Eq.~\eqref{eq:Lu}--\eqref{eq:Tigan}, respectively.
    The black circle with coordinates $(A=0.608,\,\lambda=0.669)$ corresponds to the original Lorenz attractor [see phase portrait in Fig.~\ref{fig:UNF_bif_diag_Lorenz}(e)]. The red circle with coordinates $(A=1.29,\,\lambda=0.26)$ corresponds to the original Chen attractor [see phase portrait in Fig.~\ref{fig:UNF_bif_diag_Chen}(c)].}
    \label{fig:UNF_systems_diag}
    \vspace*{12pt}
\end{figure*}

\begin{remark} The generalized system~\eqref{eq: generalized Lorenz} is not in normal form because one of its four parameters is redundant.
\end{remark}

\section{Nonlocal analysis of the saddle manifolds}


UNF~\eqref{eq: main system} has three equilibria $O(0,0,0)$ and $E^{\pm}\left(\pm\sqrt{\frac{\alpha}{\alpha+\beta}},\,0,\,\frac{\beta}{\alpha+\beta}\right)$.
   Two symmetrical equilibria $E^{\pm}$ are stable in the region of parameters 
   \begin{equation}
   \begin{array}{cc}
     \lambda>\lambda_{s}(\alpha,\beta), \\
     \\
    \lambda_{s}(\alpha,\beta)\triangleq \frac{1}{2(\alpha+\beta)}\left(\sqrt{m^2 + 8\beta(\alpha+\beta)}-m\right),
   \end{array}
    \label{eq: two stable equilibria} 
   \end{equation}  
where $m=2+\alpha\beta + \alpha^2$.
At $\lambda=\lambda_{s}(\alpha,\beta)$ Andronov-Hopf bifurcation occurs and for  $\lambda<\lambda_{s}(\alpha,\beta)$ the equilibria $E^{\pm}$ become the saddle-foci having 2D unstable manifolds which simultaneously surves the stable manifolds of two saddle limit cycles appeared after the bifurcation.

Characteristic exponents of the equilibrium $O$ are 
\begin{equation}
      e_{1,2}=-\dfrac{\lambda}{2}\pm\sqrt{\dfrac{\lambda^2}{4}+1}, \quad
      e_3=-\alpha<0.
\label{eq: eigenvalues}
\end{equation}
Therefore the equilibrium $O$ is the saddle having 1D unstable $W^u$ and 2D stable $W^s$ manifolds. 
The Shilnikov's condition $|e_3|<e_1$ takes the form \cite{chua2001methods}
\begin{equation}
    \sigma\triangleq \frac{1-\alpha^2}{\alpha}-\lambda>0.
    \label{eq: Shilnikov's condition}
\end{equation}
Since $|e_3|<e_1<|e_2|$ under this condition eigenvector $v_3(0,0,1)$ corresponding to $e_3$ determines the leading direction.

\subsection{1D invariant manifold $W^u$
    \label{Sec:1D invariant manifold $W^u$}}
\noindent Consider the Poincar\'{e} cross-section $S=S^-\cup S^+$, where 
\begin{equation*}
\begin{array}{c}
    S^{-}=\{ x<0,\,y=0,\,z>0\},\\
    S^{+}=\{ x>0,\,y=0,\,z>0\}.
\end{array}
\end{equation*}
\begin{lemma}
      Manifold $W^u$ passes through the region 
\begin{equation*}
g^u=\left\{ 0\leq y< \sqrt{x^2-\frac{x^4}{4}},\;\frac{\beta}{\alpha+2}x^2<z<\frac{\beta}{\alpha}x^2 \right\} 
\end{equation*}
and intersects the cross-section $S^{+}$ at the point
\begin{equation}
        p^{u}(x^u,z^u)=W^u\cap S^{+},
\label{eq:intersects the cross-section in a point}
\end{equation}
such that 
\begin{equation*}
0<x^u<\sqrt{2},\quad 0<z^u<\dfrac{2\beta}{\alpha}.
\end{equation*}
\label{lem:(1)intersection of manifolds}
\end{lemma}

\begin{proof}
   The Lyapunov directing function of the form 
    \begin{equation}
        V=\frac{x^4}{4}-\frac{x^2}{2}+\frac{y^2}{2}
        \label{eq: Lyapunov function}
    \end{equation} has the derivative with respect to UNF~\eqref{eq: main system} 
    \begin{equation}
       \dot{V}=-yz-\lambda y^2<0, \ \ yz>0. 
       \label{eq: Lyapunov function derivative}
    \end{equation}
    Then the trajectories of UNF~\eqref{eq: main system}  enter the region $V<0$ through its boundary for $yz>0$. Besides, the trajectories of UNF~\eqref{eq: main system}   enter the region $g^u$ through the boundaries $z=\frac{\beta x^2}{\alpha+2}$ and $z=\frac{\beta x^2}{\alpha}$ (see Lemma 3 in \cite{belykh1984bifurcation}). In a small neighborhood of the origin, the equation of $W^u$ can be written as follows 
    \begin{equation}
        y=e_1x, \ \ z=\frac{\beta}{\alpha+2e_1}x^2.
        \label{eq: small neighborhood}
    \end{equation}
    Hence $W^u$ enters $g^u$ and since $\dot{x}=y>0$ the manifold $W^u$ intersects $S^{+}$ at the point $p^{u}=W^u\cap S^{+}\in g^u |_{y=0}$.
\end{proof}

\subsection{2D invariant manifold $W^s$
\label{2D invariant manifold $W^s$}}
\noindent Here we present improved version of the nonlocal analysis of $W^s$ given in \cite{belykh1984bifurcation}.
First we remind that the line 
\begin{equation*}
I=\Big\{ x=y=0,\ z\geq 0\Big\}
\end{equation*}
is invariant and lies in $W^s,\, I\subset W^s$.
Due to involution $\left(x,y,z\right) \to \left(-x,-y,z\right)$ of UNF~\eqref{eq: main system}, it is sufficient to study $W^s$ in the region $x\geq0$. Moreover, since $\dot{z}>0$ for $z<0$ all trajectories leave the region $z<0$.
Then we consider $W^s$ in the region $\left(x\geq 0,\ z\geq 0\right)$.

\subsubsection{Tangent stable manifold $W^s_t$
\label{subsec:Tangent stable manifold $W^s_t$}}

Introduce small $\varepsilon$-neighborhood of $I$
\begin{equation*}
    U_\varepsilon(I)=
    \left\{x^2+y^2<\varepsilon^2,\,z\in\mathbb{R}^1\right\}
\end{equation*}
for which UNF~\eqref{eq: main system} has the form
\begin{equation}
\begin{array}{l}
    \dot{x}=y,\\
    \dot{y}=-(z-1)x-\lambda y,\\
    \dot{z}=-\alpha z.
\end{array}
    \label{eq: main system in small neighborhood}
\end{equation}

The last equation in~\eqref{eq: main system in small neighborhood} has the solution $z(t)=z_0e^{-\alpha t}$. 
Then we can consider the system~\eqref{eq: main system in small neighborhood} as non-autonomous 2D system with  additional parameter $z_0$.
This system has globally stable manifold $\{z=0\}$ and linear system at it gives the eigenvector $v_2(1,e_2,0)$ defining the line $y=e_2 x$ tangent to the stable manifold $W^s$.
The invariant line $I$ of the system~\eqref{eq: main system in small neighborhood} has stable invariant manifold $W^{s}_{t}$ which is tangent to $W^s$.

\begin{proposition}
    The surface in the phase space $(x,y,t)$ of the form
\begin{equation}
    W_{\eta}=\left\{x,y,t \ \Big| \ y=-\left(\frac{\lambda}{2}+\eta(t)\right)x, \ t\in \mathbb{R}^1\right\}
    \label{eq: surface in the phase space}
\end{equation}
is an integral manifold of the system~\eqref{eq: main system in small neighborhood} if the function $\eta(t)$ is a solution of Riccati equation
\begin{equation}
\begin{array}{l}
     \dot{\eta}=a(t)+\eta^2, \\
     a(t)= z_0e^{-\alpha t}-1-\dfrac{\lambda^2}{4}.
\end{array}
    \label{eq: Riccati equation}
\end{equation}
\label{prop:(3) integral manifold }
\end{proposition}
\begin{proof}
    The derivative of the function $w=y+\left(\frac{\lambda}{2}+\eta(t)\right)x$ with respect to the system~\eqref{eq: main system in small neighborhood} at $w=0$
    \begin{equation*}
        \dot{w}\Big|_{w=0}=\left(\dot{\eta} - a(t)-\eta^2\right)x
    \end{equation*}
    is identical zero at any solution of the equation~\eqref{eq: Riccati equation}. This implies that  $W_{\eta}$ is an integral manifold of the system~\eqref{eq: main system in small neighborhood}.
    \end{proof}

    Note that the parameter $z_0$ and initial condition $\eta(0)=\eta_0$ define continuum of solutions of the equation~\eqref{eq: Riccati equation}, and each of them generates the integral manifold $W_{\eta}$.
    Hence, our goal is to find a unique solution of~\eqref{eq: Riccati equation} defining the invariant manifold $W^{s}_{t}$.

Below we show that this manifold has two regions: saddle and node region where the manifold does not twist and the focus region where it rotates around $I$. 

\subsubsection{Saddle and node piece of $W^{s}_{t}$
\label{subsec:Saddle and node piece of $W^{s}_{t}$}}

    In trivial case $z_0=0$ the equation~\eqref{eq: Riccati equation} becomes $\dot{\eta}=-\left(1+\frac{\lambda^2}{4}\right)+\eta^2$, 
    and its equilibrium point $\eta=\sqrt{1+\frac{\lambda^2}{4}}$ defines the line~\eqref{eq: surface in the phase space}.
    This line coincides with the line $y=e_2x$ along the eigenvector $(1,e_2)$ of the saddle $O$.
Hence we need to find a bounded solution of the equation~\eqref{eq: Riccati equation} satisfying the condition 
\begin{equation}
    \eta(0)=\eta_0, \  z_0=1+\frac{\lambda^2}{4}, \  \displaystyle{\lim_{t \to \infty}} \eta(t)=\sqrt{1+\frac{\lambda^2}{4}},
    \label{eq: bounded solution}
\end{equation}
which defines the $W^s_t$ at the interval $z\in\left[0,1+\frac{\lambda^2}{4}\right]$. Due to~\eqref{eq: Riccati equation},~\eqref{eq: bounded solution}  the function $a(t)$ satisfies the condition 
\begin{equation*}
   -1-\frac{\lambda^2}{4}<a(t)<0, \ \ t\in [0,\infty).
\end{equation*}
This implies that we have the differential inequalities 
\begin{equation*}
    -\left(1+\frac{\lambda^2}{4}\right)+\eta^2< \dot{\eta}<\eta^2,
\end{equation*}
from which it follows the existence of bounded solution of the equation~\eqref{eq: Riccati equation}
lying at the interval $\left(0,\sqrt{1+\frac{\lambda^2}{4}}\right)$. This solution $\eta_s(t)$ is unstable and unique at this interval since the variational equation for $\eta=\eta_s(t)>0$
\begin{equation*}
    \dot{u}=2\eta_s(t)u
\end{equation*}
defines the positive Lyapunov exponent.

Since $a(t)$ decreases to $ -\left(1+\frac{\lambda^2}{4}\right)$ exponentially the condition~\eqref{eq: bounded solution} is valid and~\eqref{eq: surface in the phase space} for $\eta=\eta_s(t)$ is the equation of the tangent manifold $W^s_t$ of the form
\begin{equation*}
\begin{array}{c}
 W^s_{t}=\\
 \left\{ y=-\left(\frac{\lambda}{2}+\eta_s(t)\right)x,\;
     z=\left(1+\frac{\lambda^2}{4}\right)e^{-\alpha t},\; t>0\right\}.
\end{array}
\end{equation*}

\subsubsection{Focus piece of $W^s_{t}$
\label{subsec:Focus piece of $W^s_{t}$}}

Zero equilibrium of degenerate 2D system~\eqref{eq: main system in small neighborhood}, $\alpha=0$, for $z>1+\frac{\lambda^2}{4}$ becomes stable focus. The continuation of the tangent stable manifold $W^s_t$ into the region $z>1+\frac{\lambda^2}{4}$ is given by the solution of the equation~\eqref{eq: Riccati equation} in reverse time with  the initial condition 
$\eta_s(0)=\eta_0,\ \ \ z_0=1+\frac{\lambda^2}{4}$. That is the equation 
\begin{equation}
\begin{array}{c}
     \dot{\eta}=-\omega^2(t)-\eta^2, \quad \eta_s(0)=\eta_0,\\
     \\
     \omega(t)=\sqrt{\left(1+\dfrac{\lambda^2}{4}\right)\left(e^{\alpha t}-1\right)}.
\end{array}
\label{eq: Riccati equation in reverse time}
\end{equation}
Denoting this solution $\eta=\eta_f(t)$ we obtain the focus piece of $W^s_t$ in the form 
\begin{equation}
\begin{array}{c}
     W^s_{t}=\left\{ y=-\left(\dfrac{\lambda}{2}+\eta
     _f(t)\right)x,\;
     z=\left(1+\dfrac{\lambda^2}{4}\right)e^{\alpha t}\right\},\\ \eta_f(0)=\eta_s(0).
\end{array}
\label{eq: focus piece of W}
\end{equation}

Then we give a qualitative estimate of the solution $\eta=\eta_f(t)$, characterizing the form of $W_t^s$. The function $\omega^2(t)$ increases exponentially, hence for any $\omega(t_1)>0$
\begin{equation}
    \omega (t)> \omega(t_1), \ t>t_1.
    \label{eq: omega increases}
\end{equation}
Considering $\omega(t)=\omega(t_1)\triangleq\omega_1=\textrm{const}$ for $t\geq t_1$, i.e. ``freezing'' $\omega(t)$, we obtain a solution of the integrable in this case equation~\eqref{eq: Riccati equation in reverse time} in the form
\begin{equation}
    \eta(t)=-\omega_1 \tan (\omega_1 t).
    \label{eq:  solution  Riccati equation in reverse time}
\end{equation}
Next we choose the value $t_1$ such that 
\begin{equation*}
    \omega_1 t_1=-\frac{\pi}{2},
\end{equation*}
corresponding to blow up of the solution~\eqref{eq:  solution  Riccati equation in reverse time} $\displaystyle{\lim_{t \to t_1}} \ \eta(t)\to \infty,$ defining the vertical line of~\eqref{eq: focus piece of W}.
Then for $t>t_1$ the function~\eqref{eq:  solution  Riccati equation in reverse time} is decreasing up to the next blow up at $t=t_1+\frac{\pi}{\omega_1}$ when $\displaystyle{\lim_{t \to t_1+\frac{\pi}{\omega_1}}} \ \eta(t)=- \infty.$

Due to~\eqref{eq: omega increases} the solution of~\eqref{eq: Riccati equation in reverse time} reaches blow up earlier then~\eqref{eq:  solution  Riccati equation in reverse time}  at $t_2=t_1+\triangle t_1,\ \ \triangle t_1<\frac{\pi}{\omega_1}.$  Repeating the same procedure replacing $t_1$ with $t_2$ and then $t_k$ with $t_{k+1}$ we get a recurrent infinite sequence 
\begin{equation}
    t_{k+1}=t_k+\triangle t_k=t_1+\sum_{i=1}^{k} \triangle t_i, \ \ k=1,2,\dots,
    \label{recurrent infinite sequence}
\end{equation}
where
\begin{equation}
    \triangle t_k<\frac{\pi}{\omega_k}, \quad \omega_k\triangleq\omega(t_k), \ \ k=1,2,\dots 
    \label{delt t}
\end{equation}
According to~\eqref{eq: Riccati equation in reverse time}, the number series  $\displaystyle{\sum_{i=1}^{\infty}\frac{\pi}{\omega_i}}$ converges.
Hence due to \eqref{delt t} the value $\displaystyle{\sum_{i=1}^{\infty}\triangle t_i}$ is bounded, and the infinite sequence~\eqref{recurrent infinite sequence}  has finite limit $t_{\infty}=T$. Thus the function $\eta_f(t)$ in~\eqref{eq: focus piece of W} decreases from $\infty$ up to  $-\infty$ at each interval $(t_k,t_{k+1}), \ k=1,2,\dots$. Hence there exists a sequence of time value $\tau_k\in(t_k,t_{k+1})$ such that 
\begin{equation}
    \frac{\lambda}{2}+\eta_f(\tau_k)=0,\ \ \ k=1,2,\dots,
    \label{sequence of time value}
\end{equation}
corresponding to a family of horizontal lines in $W^s_t$ in the planes $\left\{x,\,y,\,z_k=\left(1+\frac{\lambda^2}{4}\right)e^{\alpha \tau_k}\right\}$.

Hereby the focus piece of tangent stable manifold is formed by the straight line~\eqref{eq: focus piece of W} moving in reverse time along $z-$axis up to the limit
\begin{equation*}
z^{*}=\left(1+\frac{\lambda^2}{4}\right)e^{\alpha T},
\end{equation*}
at the same time winding around $z-$axis  with increasing angle velocity $\omega(t)$ and therefore is the \textit{ruled double helical surface.}

\subsubsection{Nonlocal geometry of $W^s$
\label{subsec:Nonlocal geometry of $W^s$}}

We denote the curve 
\begin{equation*}
    \eta^s(t)=\left\{
\begin{array}{ll}
\eta_s(t),   &\textrm{\;for\quad} 0<z\leq1+\dfrac{\lambda^2}{4}, \\
 \eta_f(t),     &\textrm{\;for\quad}  1+\dfrac{\lambda^2}{4}<z<z^{*},
\end{array}
\right.
\end{equation*}
and the function 
\begin{equation*}
    \xi=y-K(t)x, \ \ K(t)\triangleq-\left(\frac{\lambda}{2}+\eta^s(t)\right),
\end{equation*}
where $K(t)$ is the slope of the straight line $\xi=0.$ Then the tangent manifold $W^s_t$ takes the form
\begin{equation}
    W^s_t=\Big\{y=K(t)x,\ \ z\in[0,z^{*})\Big\}.
    \label{tangent manifold}
\end{equation}

This manifold is a surface formed by a straight line moving along the $z-$axis making a turning around this axis without rotation in the saddle and node region $0<z\leq1+\frac{\lambda^2}{4}$ and rotating in the focus region $1+\frac{\lambda^2}{4}<z<z^{*}$.

Our goal is to describe the main features of the continuation of the stable manifold $W^s$ from the local part to its first intersection with the cross-section $S$.
We denote this first intersection $l^{s}=W^s\cap S$ which must  have two symmetrical parts $l^s=l^s_{+}\cup l^s_{-}$ such that
\begin{equation*}
l^{s}_{+}=W^s\cap S^{+}, \quad l^{s}_{-}=W^s\cap S^{-}.
\end{equation*}

\begin{proposition}
    The $x-$coordinates of the intersection $W^s\cap S$ are bounded.
\label{prop:(4) The x-coordinates of the intersection are bounded.}
\end{proposition}
\begin{proof}
    We use the 2D auxiliary system
    \begin{equation}
    \dot{x}=y, \ \ \dot{y}=-(x^2-1)x-\lambda y,
    \label{2D auxiliary system}
\end{equation}
having the integral~\eqref{eq: Lyapunov function} in the conservative case $\lambda=0$. For $\lambda>0$ this system is globally stable since the derivative~\eqref{eq: Lyapunov function derivative} is negative for $z=0$.

The stable separatices of the saddle $O(0,0)$  while being in the region $xy<0$ intersect first time the $x$-axis in the points $(x^s_0,0)$, where $x^s_0\big|_{\lambda=0}=\sqrt{2}$, and increases with the parameter $\lambda$ increase.
Hence the stable separatrix of the system~\eqref{2D auxiliary system} at the interval $x\in (0,x^s_0]$ $\Big(x\in [-x^s_0,0)\Big)$ can be given by the equation $y=y^s(x)$ $\Big(y=-y^s(-x), \textrm{respectively}\Big)$ such that $y^s<0$ for $x\in(0,x_0^s)$ and $y^s(x_0^s)=0$.

 The vector field of UNF~\eqref{eq: main system} on the surface $S^{+}_{aux}=\left\{y=y^s(x), \ z> 0, \ x\in (0,x_0^s]\right\}$ is oriented towards the region $y>y^s(x)$.

This statement follows from the fact that the difference $\delta$ of $\dot{y}-$component of the auxiliary system~\eqref{2D auxiliary system} and $\dot{y}-$component of the vector field of UNF~\eqref{eq: main system} satisfies the condition $\delta=xz>0 \ \textrm{for} \ x>0, \ z>0. $

From this it's easy to obtain that the $x-$coordinates of $W^s\cap S$  can not be out of the interval $|x|<x^s_0.$
\end{proof}

Now we will show the existence of the intersection $l^s$ and define its properties.
It turns out that the ruled surface~\eqref{tangent manifold} can be used in the study of nonlocal part of stable manifold $W^s$.
\begin{lemma}
       The stable manifold $W^s$ rotates around the invariant line $I$ faster then the tangent manifold $W^s_t$ and has a sequence of intersections with the cross-section $S$ such that the first of them is two  curves which can be written in the form 
\label{lem:(2)manifold rotation around the invariant line}
 \begin{equation}
 l^{s}_{\pm}=\left\{x=\pm x^s(z),\, z\in(0,z^{*}),\, y=0\right\}.
\label{symmetrical system}
\end{equation}
\end{lemma}

\begin{proof}
 The vector field of UNF~\eqref{eq: main system} is transversal to the tangent manifold $W^s_{t}$ and oriented towards the region 
\begin{equation}
x\xi <0.
\label{eq:xi<0}
\end{equation}
Indeed the derivative of the function $\xi$ with respect to UNF~\eqref{eq: main system} at $\xi=0$
\begin{equation}
    \dot{\xi}|_{\xi=0}=-\left(x^2+\frac{\lambda^2}{4}\right)x
    \label{derivative xi}
\end{equation}
is negative (positive) for $x>0$ ($x<0$, respectively), hence orientation of the vector field of UNF~\eqref{eq: main system} at $\xi=0$ is defined by inequality~\eqref{eq:xi<0}. 

This implies that the trajectories in $W^s$ leaving the neighborhood $U_{\varepsilon}(I)$ in reverse time can not intersect the surface $W^s_t$ and therefore get into the region $x \xi>0$ staying in it up to the moment of their intersection with the cross-section S.  Secondly, the turn and rotation of the tangent manifold $W^s_t$ make the stable manifold $W^s$ to rotate  faster then $W^s_t$, and due to~\eqref{derivative xi} the faster the larger the modulus of the $x-$coordinates of the orbits in $W^s$.

In a small neighborhood of invariant line $I$ the manifold $W^s$ is close to $W^s_t$. Therefore the nonlocal continuation of $W^s$ is also close to the continuation of intersection curves 
\begin{equation}
\begin{array}{ll}
     W^s_t\cap U_{\varepsilon}=\Big\{
 y=K(t)x,\, x^2+y^2=\varepsilon^2,\\ z=\left(1+\dfrac{\lambda^2}{4}\right)e^{-\alpha t},\, 0\leq z\leq z^{*} \Big\} 
      \label{eq:intersection curves}
\end{array}
\end{equation}
along the trajectories of UNF~\eqref{eq: main system} in reverse time.

Due to~\eqref{sequence of time value} the slope $K(t)$ changes the sign at the sequence of time $\tau_k,\ k=1,2,\dots$
such that 
\begin{equation*}
    \begin{array}{rr}
        K(t)<0,\quad& \tau_{2i}<t<\tau_{2i+1}, \\
         K(t)>0,\quad& \tau_{2i+1}<t<\tau_{2i+2},  
    \end{array}
    \quad i=0,1,2,\dots
\end{equation*}
with initial value $K<0$, $t\in(0,\tau_1)$ corresponding to the saddle and node piece $0\leq z<1+\frac{\lambda^2}{4}$ for $i=0$.
From Proposition~\ref{prop:(4) The x-coordinates of the intersection are bounded.} and~\eqref{eq:intersection curves} it follows that the trajectories in $W^s$ leaving the neighborhood $U_{\varepsilon}(I)$ in reverse time can not intersect the surface $W^s_t$ and therefore get into the region $x\xi>0$.
Since the $x-$coordinates of the trajectories in $W^s$ are bounded due to Proposition~\ref{prop:(4) The x-coordinates of the intersection are bounded.} the stable manifold $W^s$ has the first intersection with the cross-section $S$ along symmetrical curves $l_{+}^s$ and $l_{-}^s$~\eqref{symmetrical system}.

The function $x=x^s(z)$ satisfies the condition
\begin{equation}
\begin{array}{c}
    \begin{array}{rl}
    x^s\big(z(t)\big) = 0 &\textrm{ for\quad}t=\tau_k,\\
    K(t)\,x^s\big(z(t)\big) < 0 &\textrm{ for\quad}t \neq \tau_k,
   \end{array}    
\end{array}\quad k = 1,2,\dots\,.
    \label{eq:x^ssatisfies the condition}
\end{equation}
\end{proof}

\begin{figure*}
    \centering   \includegraphics[width=0.27\linewidth]{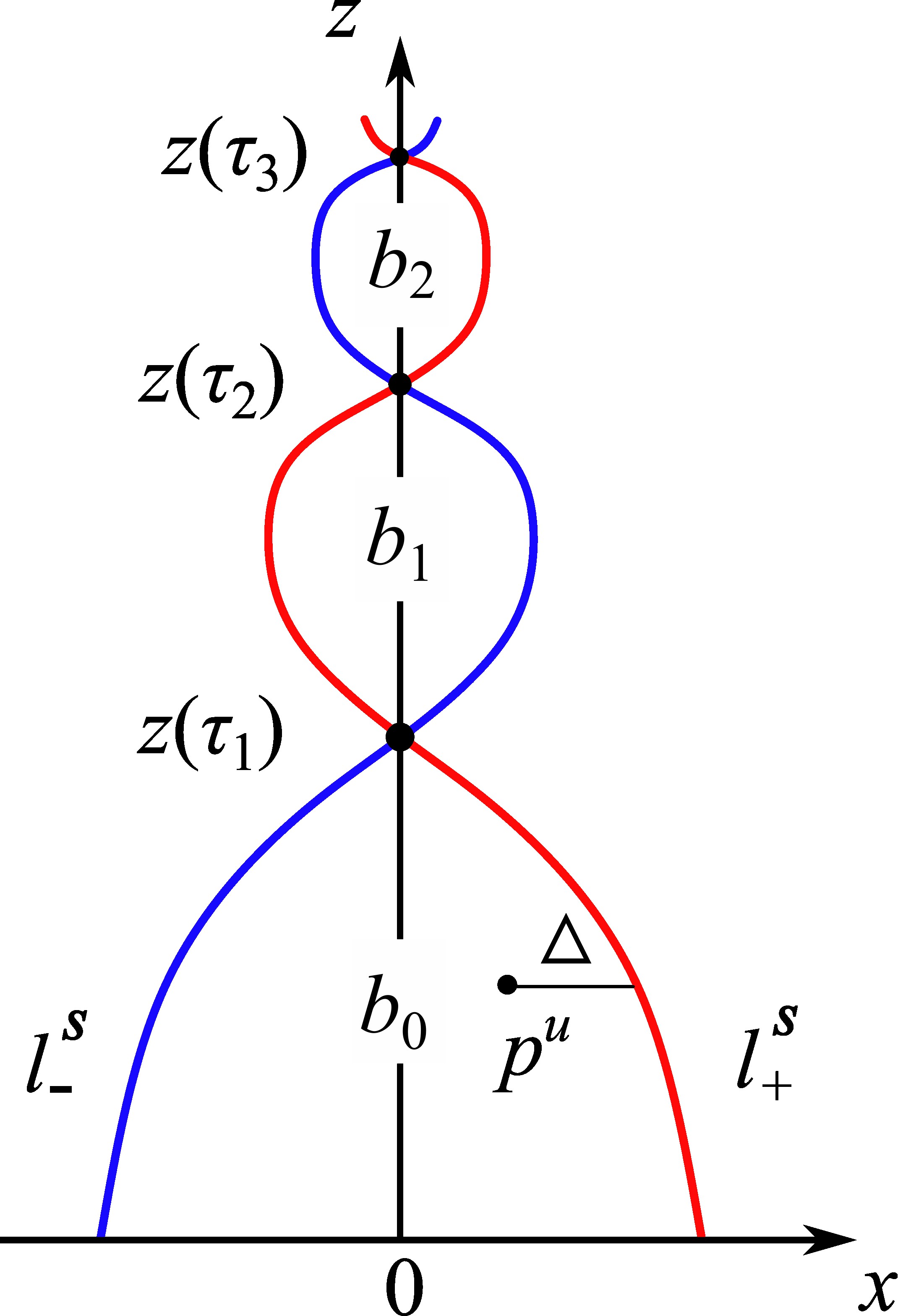}
    \smallskip
    \caption{The qualitative form of the intersection lines $l^s_{\pm}=W^s\cap S$ in Lemma~\ref{lem:(2)manifold rotation around the invariant line} and the regions $b_k$ in Corollary~\ref{cor: 5.1}. The distance $\triangle$ between $p^u=W^s\cap S$ and $l^s_{+}$ is indicated.}
    \label{fig: f2}
    \vspace*{12pt}
\end{figure*}

\begin{corollary}
Due to~\eqref{eq:x^ssatisfies the condition}, the curves $l^s_{\pm}$ form on the cross-section $S$ the infinite sequence of domains
    \begin{equation}
    \begin{array}{c}
     b_k=\Big\{|x|<|x^s(z(t)|,\, \tau
        _k\leq t\leq \tau_{k+1},\, y=0\Big\},\\
         k=0,1,2,\ldots\,.
    \end{array}
    \label{eq: sequence of domains}
    \end{equation}
    \label{cor: 5.1}
\end{corollary}

They are symmetric with respect to the invariant line $I$, adjacent domains adjoin each other at points $(0,0, z(\tau_k)).$ It is funny that the first domain $b_0$ lying in saddle and node region looks like a skirt and the rest $b_k$ look like shrinking bubbles (see Fig.~\ref{fig: f2}).

\begin{proposition}
There exists a layer $L=\{ (x,y) \in \mathbb{R}^2, \ z \in (z_1,z_2) \}$ where constants $z_{1,2}$ satisfy the condition $1<z_1<z_2$, such that any orbit of UNF~\eqref{eq: main system} lying in $L$  tends to the neighborhood $U_{\varepsilon}$ of the invariant line $I$.
\label{prop:(5) layer, where constants  tends to the neighborhood of the invariant line.}
 \end{proposition}

 \begin{proof}
     Consider the Lyapunov function 
     \begin{equation*}
         V=\frac{c_1x^2}{2}+\frac{x^4}{4}+c_2xy +\frac{y^2}{2},
         \end{equation*}
    where positive parameters $c_1$ and $c_2$ when $V$ satisfy the condition $c_1>c_2^2$ when $V$ is positive definite.
    The time derivative $\dot{V}$ with respect to UNF~\eqref{eq: main system} reads 
         \begin{equation*}
         \dot{V}=-\left(c_2(z-1)x^2-\left(C-z\right)xy +(\lambda-c_2)y^2\right)-c_2x^4,
         \end{equation*}
    where $C=c_1+1-\lambda c_2$.
    The function $\dot{V}$ is negative definite under the condition $\lambda>b$ for $z\in (z_1,z_2)$, where
    \begin{equation*}
             z_{1,2}=C+\frac{B}{2}\mp \sqrt{BC+\frac{B^2}{ 4}},\quad B=4c_2(\lambda-c_2).
    \end{equation*}
    Hence, an orbit $\Big(\tilde{x}(t),\,\tilde{y}(t),\,\tilde{z}(t)\Big)\in L$ satisfies the condition $\displaystyle{\lim_{t \to \infty}} \Big(\tilde{x}(t),\,\tilde{y}(t)\Big)=(0,0).$
 \end{proof}

\section{Limiting cases of UNF
\label{sec:Special cases}}
Consider important limiting cases of the system parameters.

\subsection{The case of small $\alpha$}

For $\alpha=0$ all point of the invariant line $I$  are equilibrium points.
For $z<1$, they are saddles; for $1<z<1+\frac{\lambda^2}{4}$, they are stable nodes; and for $z>1+\frac{\lambda^2}{4}$, they transform from nodes to stable foci.
The unstable  manifold $W^u$ of the degenerate saddle $O(0,0,0)$ due to~\eqref{eq: small neighborhood} enters the region $z>0$. 

Due to $\dot{z}=\beta x^2>0$, $z-$ coordinates of all trajectories of UNF~\eqref{eq: main system} besides $I$ increases for $\beta>0$ and reaches the invariant region $G_{+}=\{x\neq 0,\ y\neq 0, \ z>1\}$. 
From Proposition~\ref{prop:(5) layer, where constants  tends to the neighborhood of the invariant line.} it follows that each orbit of UNF~\eqref{eq: main system} enters the neighborhood $U_\varepsilon$ through a point $\left(x_0,y_0,z_0\ | \ x_0^2+y_0^2=\varepsilon^2, \ z_0>1\right)$.
Considering for simplicity the coordinate $z_0>1+\frac{\lambda^2}{4}$ and denoting $\Omega^2=z_0-\left(1+ \frac{\lambda^2}{4} \right)>0$ we introduce the solution of linearized UNF~\eqref{eq: main system} at $I$ in the form
\begin{equation}
    \begin{array}{l}
         x_e(t)=\rho e^{-\frac{\lambda}{2}t}\sin(\Omega t+\varphi),\\
         \\
         y_e(t)=-\rho\sqrt{z_0-1}e^{-\frac{\lambda}{2}t}\sin\left(\Omega t +\varphi-\arctan\frac{2\omega}{\lambda}\right),
         \end{array}
         \label{s1}
\end{equation}
where $\rho$ and $\varphi$ are given by the initial conditions 
\begin{equation}
    \begin{array}{l}
        x_0=\rho\sin \varphi, \\
         y_0=-\rho\sqrt{z_0-1}\sin\left(\varphi-\arctan\frac{2\Omega}{\lambda}\right). 
    \end{array}
    \label{s2}
\end{equation}

Then the increase of the coordinate $z$ inside $U_{\varepsilon}$ is defined by the equation 
\begin{equation}
    \dot{z}=\beta x^2_e (t).
    \label{s3}
\end{equation}
From~\eqref{s1}-\eqref{s3} we obtain the explicit solution 
\begin{equation*}
    \begin{array}{l}
        z(t)=z_0+\beta \rho^2 M(t), \\
        \\
        M(t)= \int_{0}^{t} e^{-\lambda t'}\sin^2(\Omega t'+\varphi) \,dt'. 
    \end{array}
\end{equation*}
Obviously the function $M(t)$ is bounded and
\begin{equation*}
    \displaystyle{\lim_{t \to \infty}} M(t)\triangleq M_{\infty}<\frac{1}{2\lambda}.
\end{equation*}
Hence each trajectory of UNF~\eqref{eq: main system} for $\alpha=0$ does not go to infinity and tends to one of the equilibrium states of the invariant line $I$.

    The unstable manifold $W^u$ of the degenerate saddle $O(0,0,0)$ of UNF~\eqref{eq: main system}, $\alpha=0$ tends to an equilibrium point $I_u(0,0,z_u)\in I$, where 
\begin{equation}
    z_u=z_0+\beta \rho^2M_{\infty}.
    \label{eq:z_u}
\end{equation}

The point $I_u$ as well is the limiting point for the symmetrical manifold $W_{-}^u$ defined at $O$ by~\eqref{eq: small neighborhood} for $x<0$.

\begin{proposition}
    For infinitesimal $\alpha>0$ in a small neighborhood of point $I_u$ there exists the turning point $p_0^u \in \left(W^u\cap S^{+}\right)$ after which the orbit $W^u$ begins to move downwards along the invariant line $I$ rotating around it in the focus region and leaving $U_{\varepsilon}(I)$ in the saddle region.
    \label{prop:(6) small alfa}
\end{proposition}
\begin{proof}
    For any small $\alpha>0$ the coordinate $z$ decrease along the invariant line $I$ and its neighborhood $U_{\varepsilon}(I)$ up to the saddle region.
    Hence the unstable manifold $W^u$ reaches the vicinity of the point $I_u\in U_{\varepsilon}$ and remaining in $U_{\varepsilon}$ moves down to the saddle region rotating around $z$-axis. The smaller parameter $\alpha$ the greater the number of times the curve $W^u$ rotates.
\end{proof}

\subsection{The case of $\lambda=0$}
For $\lambda=\alpha=\beta=0$, UNF~\eqref{eq: main system} becomes degenerate system having a foliation $z=\textrm{const}$ with the integral
\begin{equation*}
    H(x,y)=\frac{x^4}{4}+(z-1)\frac{x^2}{2}+\frac{y^2}{2}=\textrm{const}.
\end{equation*}
\begin{proposition}
    For $\lambda=0$, $\alpha>0$, $\beta>0$, UNF~\eqref{eq: main system} is globally unstable.  The  unstable manifold $ W^u$ goes to infinity.
\label{prop: (7) globally unstable system}
\end{proposition}

\begin{proof}
    Consider the Lyapunov function 
    \begin{equation*}
        V=H(x,y)-\frac{\alpha}{4\beta}z^2,
    \end{equation*}
    which derivative with respect to UNF~\eqref{eq: main system} reads
    \begin{equation*}
        \dot{V}=-\lambda y^2 + \frac{\dot{z}^2}{2\beta}.
    \end{equation*}
For $\lambda=0,\ \beta>0, \ -\alpha z+\beta x^2 \neq 0,$ we have inequality $\dot{V}>0$ which implies that for $t\to \infty$  the only bounded orbits are that which lie in 2D stable manifold $W^s$ of the saddle $O$ and $1D$ stable manifolds of the saddle-foci $E^{\pm}$. All the other orbits including unstable manifold $W^u$ tends to infinity.
\end{proof}


\subsection{The case of large $\lambda$}
For large damping $\lambda\gg 1$ we introduce new parameter and rescaling
\begin{equation}
    \mu = \lambda^{-2},\quad \lambda y \to y, \quad \lambda t \to t.
    \label{eq:rescaling_lambda}
\end{equation}
 Then UNF~\eqref{eq: main system} takes the form of slow-fast system
\begin{equation*}
\begin{array}{l}
        \dot{x}=\mu y,\\
        \dot{y}=-(x^2+z-1)x-y,\\
        \dot{z}=\sqrt{\mu} (-\alpha z+\beta x^2).
\end{array}
\end{equation*}
For small $\mu$ this system has stable slow manifold 
\begin{equation*}
    y=-(x^2+z-1)x+\dots,
\end{equation*}
and the system of equation on it has the form 
\begin{equation*}
    \begin{array}{l}
        \dot{x}=-\mu(x^2+z-1)x+\ldots,\\
         \dot{z}=\sqrt{\mu} (-\alpha z+\beta x^2).
    \end{array}
\end{equation*}
The saddle $O(0,0,0)$ and two stable foci $E^{\pm}$ are the only non wandering set of this system such that two symmetric unstable manifold $W^s$ of the saddle $O$ are twisted onto foci $E^{\pm}$ [see Fig.~\ref{fig:Cases_3_4}(a)].

\begin{figure*}
    \centering
    (a)\includegraphics[width=0.35\linewidth]{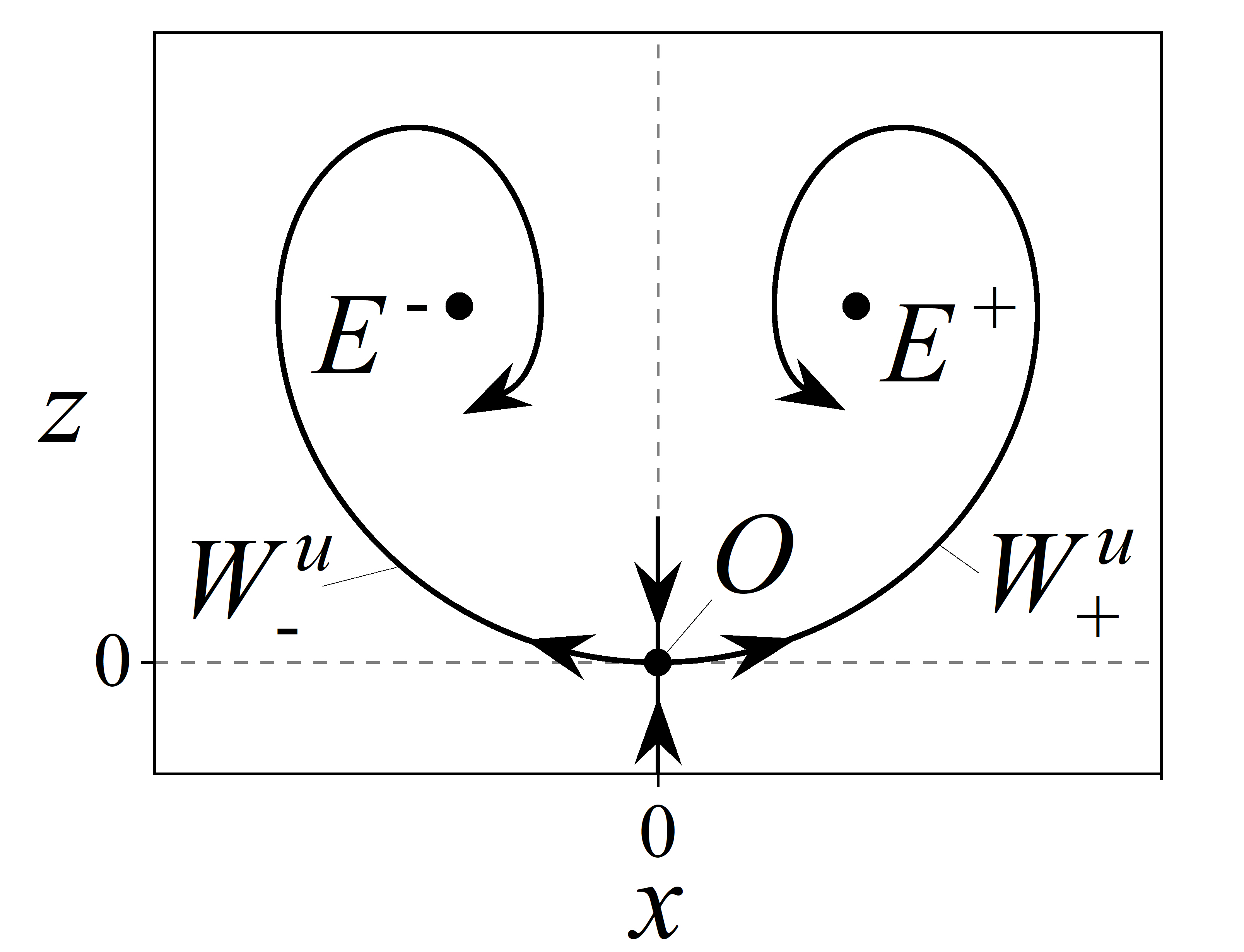}    (b)\includegraphics[width=0.35\linewidth]{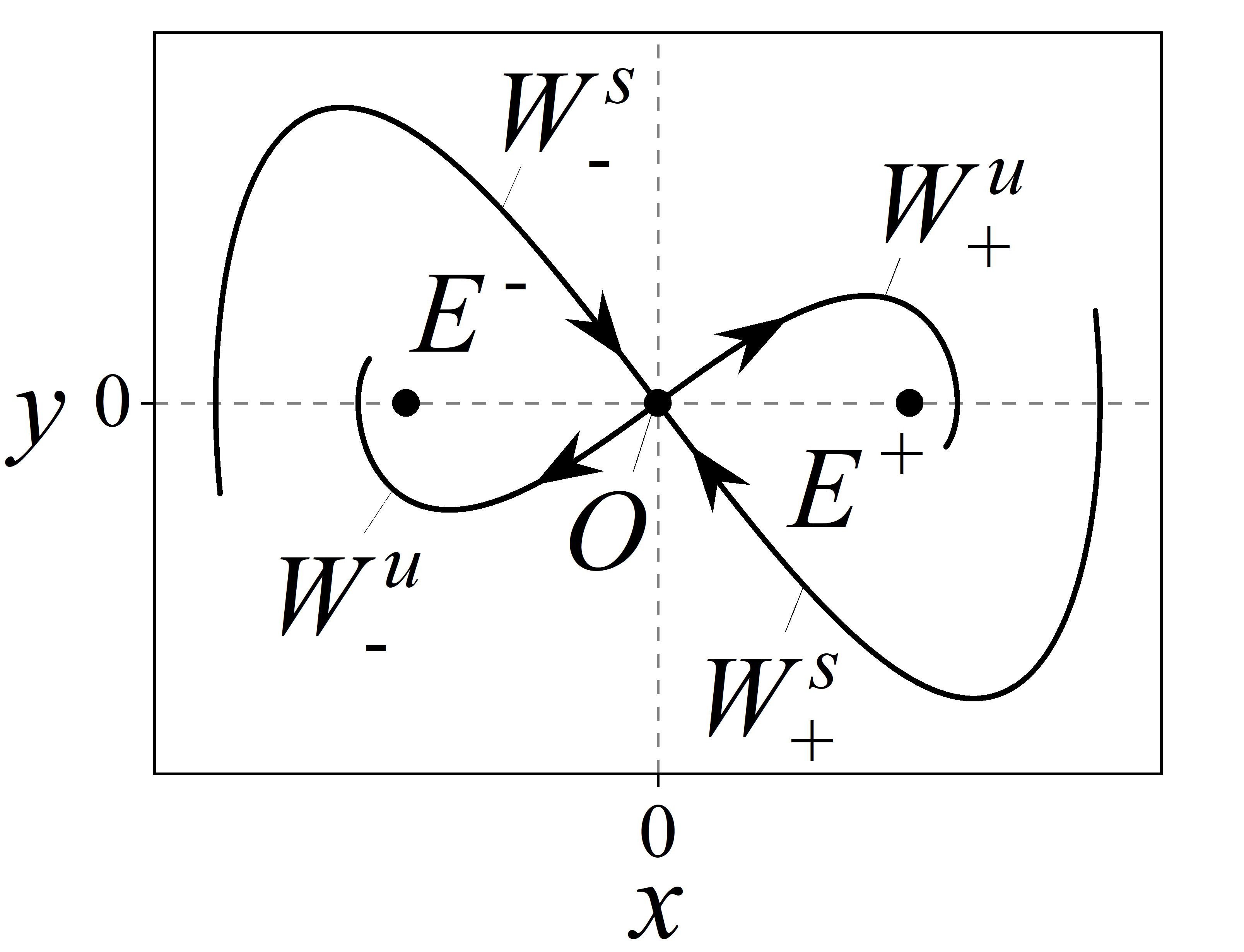}
    \caption{Phase pictures for the cases of large $\lambda$ [panel (a)] and large $\alpha$ [panel (b)].}
    \label{fig:Cases_3_4}
    \vspace*{12pt}
\end{figure*}

\subsection{The case of large $\alpha$}
For large $\alpha\gg1 $ similarly to~\eqref{eq:rescaling_lambda} we introduce small parameter and time rescaling
\begin{equation*}
    \mu=\frac{1}{\alpha},\quad \alpha t \to t.
\end{equation*}

Then UNF~\eqref{eq: main system} takes the from of  another slow-fast system
\begin{equation*}
\begin{array}{l}
     \dot{x}=\mu y, \\ 
     \dot{y}=-\mu \left((x^2+z-1)x-\lambda y\right), \\
     \dot{z}=-z+\mu\beta x^2.
\end{array}
\end{equation*}

For small $\mu$ this system has stable slow manifold $z=0$ dynamics on which is given by the system
\begin{equation}
    \dot{x}=\mu y,\quad \dot{y}=-\mu \left( (x^2-1)x-\lambda y\right),
    \label{eq: dynamics on which is given by the system}
\end{equation}
which is globally stable such that the unstable manifold $W^u$ of the saddle $O$ tends to the stable equilibria $E^{\pm}$ [see Fig.~\ref{fig:Cases_3_4}(b)]. 

\subsection{The case of large $\beta$}
For limiting case of large $\beta\gg 1$, we introduce the following new variables and parameters 
\begin{equation*}
    \sqrt{\frac{\beta}{\alpha}}x \to\ x,\quad \sqrt{\frac{\beta}{\alpha}}y \to y, \quad \mu=\frac{\alpha}{\beta}.
\end{equation*}
Then instead of UNF~\eqref{eq: main system} we obtain the system 
\begin{equation}
    \begin{array}{l}
    \dot{x}=y,\\
    \dot{y}=-(\mu x^2 + z-1)x-\lambda y,\\
         \dot{z}=\alpha(-z+x^2).
    \end{array}
    \label{system 2}
\end{equation}
UNF~\eqref{eq: main system} in this form was used in many papers (see \cite{shil1997homoclinic}, etc).
In particular, for $\beta\to \infty$ when $\mu\to 0$, the system~\eqref{system 2} becomes the well-known Shimizu-Morioka system \cite{shimizu1980bifurcation}.

\section{Homoclinic bifurcations and attractors}

\subsection{Bifurcations of homoclinic orbits
\label{sec:Bifurcations of homoclinic orbits}}

Lemmas~\ref{lem:(1)intersection of manifolds} and \ref{lem:(2)manifold rotation around the invariant line} define the first intersections $p^u=W^u\cap S^{+}$ and $l^s_{+}=W^s\cap S^{+}$  for any point of nonnegative parameters $(\alpha,\beta,\lambda)$. 

We say that the bifurcation of homoclinic orbit of the saddle $O$ is \textit{principal} if $p^u\in l^s_{+}$.

Using~\eqref{eq:intersects the cross-section in a point},~\eqref{symmetrical system} and~\eqref{eq:x^ssatisfies the condition} we introduce the function
\begin{equation*}
    \triangle_k(z^u)=x^u-x^s(z^u),\ \ z^u\in (z_k,z_{k+1}),
\end{equation*}
characterizing the mutual arrangement of the intersections $p^u$ and $l^s_{+}$ such that 
\begin{equation*}
\begin{array}{ll}
     \triangle_k=0 & \textrm{ for}\quad p^u\in l^s_{+},\\
      \triangle_k<0\;\big(>0\big) & \textrm{ for}\quad p^u\in b_k\quad\big(p^u\notin b_k, \textrm{ resp.}\big),\\
\end{array}
\end{equation*}
for $k=0,1,2,\ldots\,$. 
If the function $\triangle_k$ continuously depends on the parameters, the homoclinic bifurcation can be fined from the equation $\triangle_k=0$.
\begin{theorem}
For each point of positive parameters $(\alpha,\beta)$, 
    there exists a bifurcational surface $\lambda=\lambda_0 (\alpha,\beta)$ satisfying the conditions
    \begin{equation}
        \lambda_0 (0,0)=0,\quad \lambda_0(\alpha,\beta)>0,
    \label{eq: (71)  satisfying the  condition}
    \end{equation}
and corresponding to the existence of two symmetric homoclinic orbits of the saddle $O$ of UNF~\eqref{eq: main system} in the region $0<z<z_1$, where $z_1$ is defined in Proposition~\ref{prop:(5) layer, where constants  tends to the neighborhood of the invariant line.}. 
\label{th: (1) bifurcational surface}
\end{theorem}

\begin{proof}
    For large $\lambda\gg 1$, UNF~\eqref{eq: main system} is globally stable [see Fig.~\ref{fig:Cases_3_4}(a)] and the location of $p^u$ and $l^s_{+}$ is such that $p^u \in b_0$, i.e. $\triangle_0<0$ in this case.
    On the other hand due to Proposition~\ref{prop: (7) globally unstable system} for small $\lambda\ll 1$ UNF~\eqref{eq: main system} is globally unstable and hence $p^u\notin b_0$, i.e. $\triangle_0>0$ in this case.
    Hence, there exists a value $\lambda=\lambda_0(\alpha,\beta)$ satisfying~\eqref{eq: (71)  satisfying the  condition} for which $\triangle_0=0$ and UNF~\eqref{eq: main system} has the homoclinic orbit of the saddle $O$  in the region $x>0,$ and the symmetrical one in the region $x<0$ [see Fig.~\ref{fig:Oriented_non_oriented}(a)].
\end{proof}

\begin{figure*}
    \centering         (a)\includegraphics[width=0.35\linewidth]{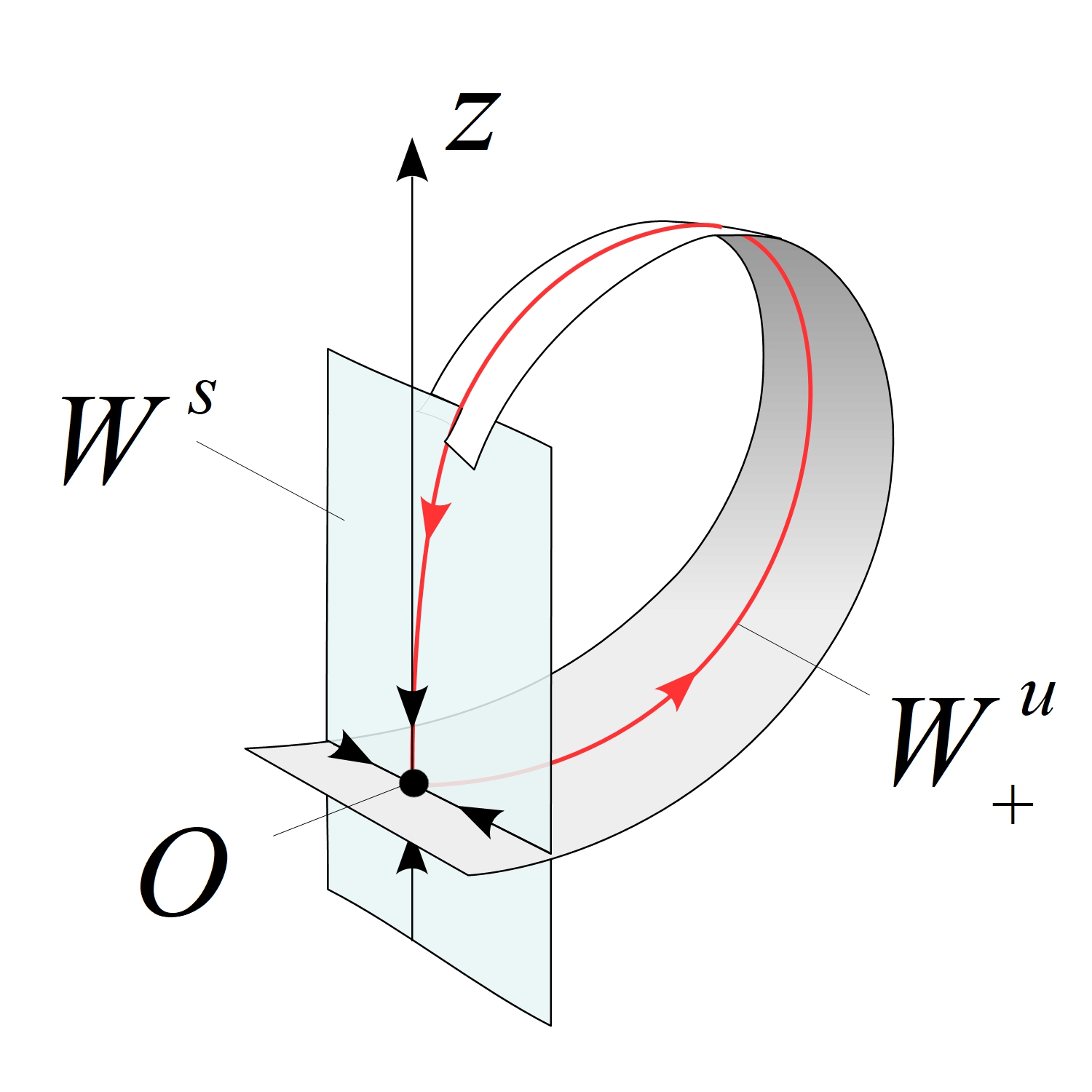}  (b)\includegraphics[width=0.35\linewidth]{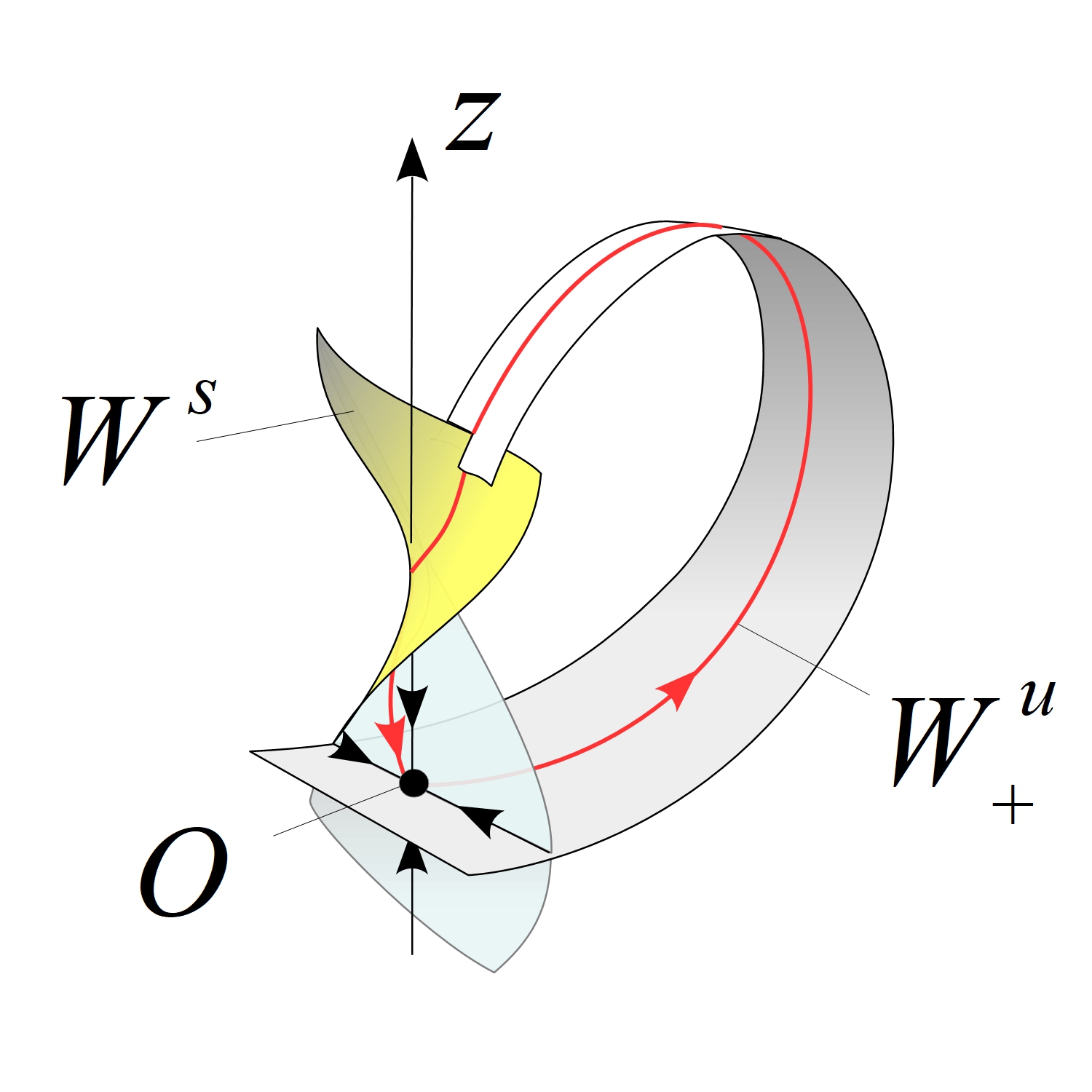}
    \caption{Qualitative pictures of topologically different homoclinic orbits. (a) Oriented case: the unstable manifold $W_+^u$ hits stable manifold $W^s$ at the region $b_0$. (b) Non-oriented case: $W_+^u$ hits $W^s$ at the region $b_1$.}
\label{fig:Oriented_non_oriented}
    \vspace*{12pt}
\end{figure*}

\begin{theorem} For bounded $\lambda>0$ and $\beta>0$,
there exists an infinite set of bifurcational surfaces $\alpha=\alpha_k(\lambda,\beta)$, $k=1, 2,\ldots$, corresponding to the existence of two symmetric twisted homoclinic orbits of the saddle $O$ of UNF~\eqref{eq: main system}. Each of them makes $m$ rotations around the invariant line $I$ for $k=2m$, $m=1,2,\ldots$, and makes $m$ and a half turns for $k=2m+1$, $m=0,1,\ldots$.
\label{th: (2) set of bifurcational surfaces}
\end{theorem}
\begin{proof}
    Due to Lemma~\ref{lem:(1)intersection of manifolds} the invariant line $I$ crosses all the domains $b_k$, $k=1,2,\dots$, having the 
    form~\eqref{eq: sequence of domains} for any $\alpha>0.$ Then according to the Proposition~\ref{prop:(6) small alfa} for small $\alpha$ the turning point $p^u_0$ falls 
    into one of the regions $b_k$, and $p^u_0\in b_k$.
    On the other hand for large $\alpha\gg 1$ the arrangement of the manifold $W^u$ and $W^s$ is defined by the globally stable system~\eqref{eq: dynamics on which is given by the system} [see  Fig.~\ref{fig:Cases_3_4}(b)]. 
    Therefore for $\alpha\gg 1$ we have $p_0^u\notin b_k$. Hence, the  function $\triangle_k$  has different signs for $\alpha\ll 1$ and $\alpha\gg 1.$ Then if the function $\triangle_k$ continuously depends on $\alpha$ there exists a value $\alpha_k(\lambda,\beta)$ such that $\triangle_k\big|_{\alpha_k}=0$ corresponding to a homoclinic orbit of the saddle $O$ having $m$ scrolls around $I$.
    
    From~\eqref{eq:z_u} and $\dot{z}\approx \beta x^2$ for $\alpha\ll 1$ it follows that the value $z_u$ increases with increasing $\beta$. Therefore $z$-coordinate of the $p^u_0$ also increases since the point $p^u_0$ is close to the point $(0,0,z_u)$ for infinitesimal $\alpha$.
    On the other hand, the coordinates $z(\tau_k)\in I$ of the infinite sequence of domains $b_k,\ k=1,2,\dots$ do not depend on the parameter $\beta$.
    Then with increasing $\beta$ for small $\alpha$, point $p^u_0$ falls in all domains $b_k$, $k=1,2,\ldots$.
    This implies that the number of homoclinic orbits is infinite if the functions $\triangle_k,\ k=1,2,\dots$ are continuous [see Fig.~\ref{fig:Oriented_non_oriented}(b) for the first non-oriented case $k=1$]. 
    The matter is that the function $\triangle_k$ is discontinuous when the manifold $W^u$ becomes heteroclinic orbit, i.e. $W^u$ hits the stable manifold of either equilibria $E^{\pm}$ or a saddle periodic orbit.
    In this case the point $p^u_0$ disappears and the function $\alpha=\alpha_k(\lambda,\beta)$ becomes a heteroclinic bifurcational surface. 
\end{proof}

\begin{remark}
    Leonov's ``fishing principle'' \cite{leonov2012general,leonov2015differences} should be corrected because the point $p^u$ of the first intersection of unstable 1D manifold $W^u$ with its cross-section $S$ (``pond'') in Leonov's case can disappear via a heteroclinic bifurcation which can arise when $W^u$ hits the stable manifold either of a saddle cycle or an equilibrium.
    The possibility of such a scenario was taken into account in the paper \cite{belykh2000homoclinic}, Theorem 3.
\end{remark}

\subsection{Lorenz-like attractors of UNF
\label{sec:Lorenz-like attractors of UNF}}

In recent papers \cite{belykh2019lorenz,barabash2025analytically} we proved the existence of the set of principal bifurcations of birth, transformation and death of Lorenz-type attractor in a piecewise-smooth system. This set is similar to well known scenario of the original Lorenz system~\eqref{eq: generalized Lorenz} for $q=1$ and positive parameters $a$, $b$ and $r$.

In the case of UNF~\eqref{eq: main system} the homoclinic bifurcation $\lambda=\lambda_0(\alpha,\beta)$ occurs in the saddle region at the transition from $p^u\in b_0$ to $p^u\notin b_0$. 
Hence this bifurcation is similar to that of original Lorenz system but now it is related both to Lorenz and Chen zones, particularly, to all systems listed in Proposition~\ref{prop:(1)The one-to-one map}.
From the Theorem~\ref{th: (1) bifurcational surface} it follows that 2D saddle manifold $W^{su}$ of the saddle $O$ based on the eigenvectors for $e_2$ and $e_3$ in~\eqref{eq: eigenvalues} forms a surface close to an annulus in a neighborhood of homoclinic orbit [see Fig.~\ref{fig:Oriented_non_oriented}(a)].
Then the homoclinic orbit $W^{su}\cap W^s$ is oriented one. The bifurcational scenario in this case is as follows [see  the bifurcation route (b)-(e) and the corresponding phase pictures in Fig.~\ref{fig:UNF_bif_diag_Lorenz}].

Under Shilnikov condition~\eqref{eq: Shilnikov's condition} Cantor set  of orbits appears after the first principal bifurcation of two homoclinic orbits to the saddle $O$ [see Fig.~\ref{fig:UNF_bif_diag_Lorenz}(c)].
This set becomes the Lorenz strange attractor  after the second main bifurcation of two heteroclinic orbits which connect (asymptotically) the saddle and two symmetrical saddle cycles [see Fig.~\ref{fig:UNF_bif_diag_Lorenz}(d)]. 
The stable manifolds of these cycles bound the basins of two stable equilibria. When the cycle merge the equilibria at Andronov-Hopf bifurcation the strange attractor becomes the unique attractor of the Lorenz system [see Fig.~\ref{fig:UNF_bif_diag_Lorenz}(e)].
The structure of this  attractor changes  at an infinite set of \textit{internal} bifurcations of homoclinic orbits and therefore is called singular hyperbolic one.

We consider the transition ``homoclinic orbit'' $\to$ ``Cantor set'' $\to$ ``strange attractor'' as the general scenario of the birth of different non-conjugate chaotic attractors. 
This is directly related to the set of bifurcations $\alpha=\alpha_k(\lambda,\beta)$. 

Finally, we should present a simple case of negative saddle value for $\lambda_0(\alpha,\beta)>\frac{1-\alpha^2}{\alpha}$. Under these conditions two symmetric stable limit cycles disappear at the homoclinic bifurcation.
These cycle were born before at Andronov-Hopf bifurcation of equilibria $E^{\pm}$.

\subsection{Chen-like attractors of UNF
\label{sec: Chen-like attractors of UNF}}

According to the Theorem~\ref{th: (2) set of bifurcational surfaces} homoclinic bifurcations 
$\alpha=\alpha_k(\lambda,\beta)$ occur in the focus region at the transition from $p^u\in b_k$ to $p^u\notin b_k$. 
The twisted surface $W^s$ makes $m$ turns around $I$ for $b_{2m}$, $m=1,2,\dots$, and makes $m$ and a half turns around $I$ for $b_{2m+1}$, $m=0,1,\dots$.
Hence the homoclinic orbit $W^s\cap W^{su}$  rotates $m$ times around $I$ moving from point $p^u_0$ to the saddle $O$ [see Fig.~\ref{fig:Chen_struct} and Fig.~\ref{fig:UNF_bif_diag_Chen}(e)]. 
Homoclinic orbit for odd $k=2m+1$ is non oriented since the surface $W^{su}$ is close to closed single-sided surface, Mobius strip for $m=0$ [see Fig.~\ref{fig:Oriented_non_oriented}(b) for $\alpha=\alpha_1(\lambda,\beta)$, $m=0$].
Whereas for even $k=2m$ the surface $W^{su}$ is twisted but is close to oriented two-sided surface.

\begin{figure*}
    \centering   \includegraphics[width=0.35\linewidth]{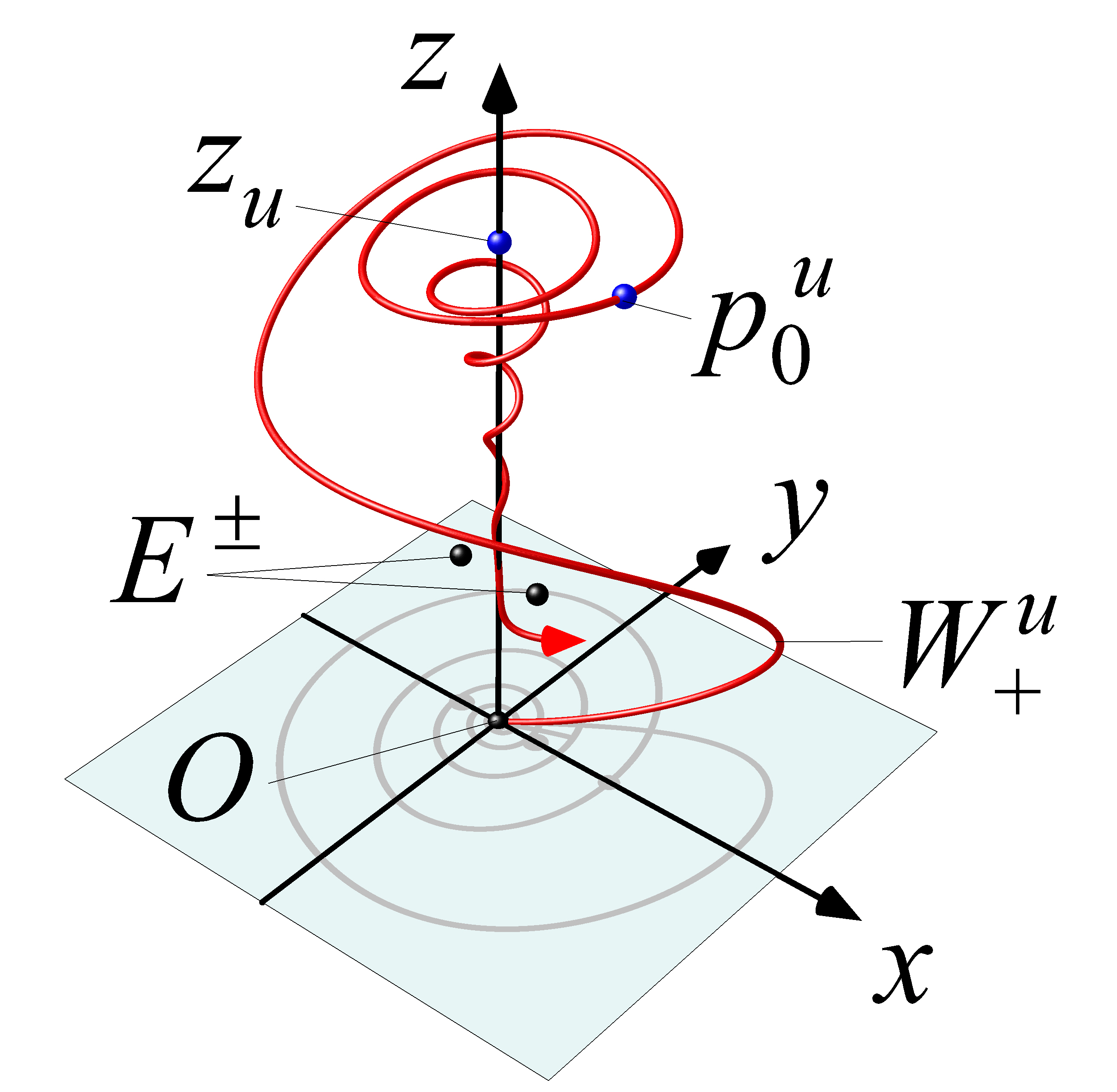}
    \smallskip
    \caption{Rotating unstable separatrix $W_+^u$ of the saddle $O$ in the vicinity of a twisted homoclinic orbit for $\alpha=0.04$, $\beta=2.47$, $\lambda=0.3$. The parameter values are selected from the region of Chen-like systems [see Fig.~\ref{fig:UNF_bif_diag_Chen}(a)].}
    \label{fig:Chen_struct}
    \vspace*{12pt}
\end{figure*}

It is well know that the Lorenz-like attractors are characterized by a sequence of two random numbers. That is the number $l$ of rotations around the left equilibrium point $E^{-}$ and the number $r$- around the right point $E^{+}$.
Chen-like attractors are characterized by a random sequence of three numbers. 
Two of them $l$ and $r$ are similar to the Lorenz case and the third one $m$ is the number of rotations around the invariant line $I$. 
Moreover, both these three numbers and the order in which they appear in the sequence are random.
Thus we get the rigorous confirmation of principal difference between Lorenz-like and Chen-like attractors established numerically via GCM approach in the paper \cite{yu2007analysis} and well illustrated in the paper \cite{wang2013gallery}.

In general, under the Shilnikov condition~\eqref{eq: Shilnikov's condition} at any homoclinic bifurcation $\alpha=\alpha_k(\lambda,\beta)$, $k=0,1,\dots$, there appears a chaotic non attracting Cantor set of orbits having different structure for odd and even $m$.
This set as was noted similarly to original Lorenz case may become a strange attractor. 
In the case when heteroclinic bifurcation occurs  (instead of homoclinic) the bifurcational scenario can be more complicated.

Note that the Cantor sets of orbits appearing at the homoclinic bifurcations preserves the topological type of homoclinic orbits defined by the number of rotations  around invariant line $I$. When the Cantor set transform into a strange attractor then these rotations preserves as well.

\begin{figure*}
 \centering
 \begin{multicols}{2}
(a)\includegraphics[width=1\linewidth]{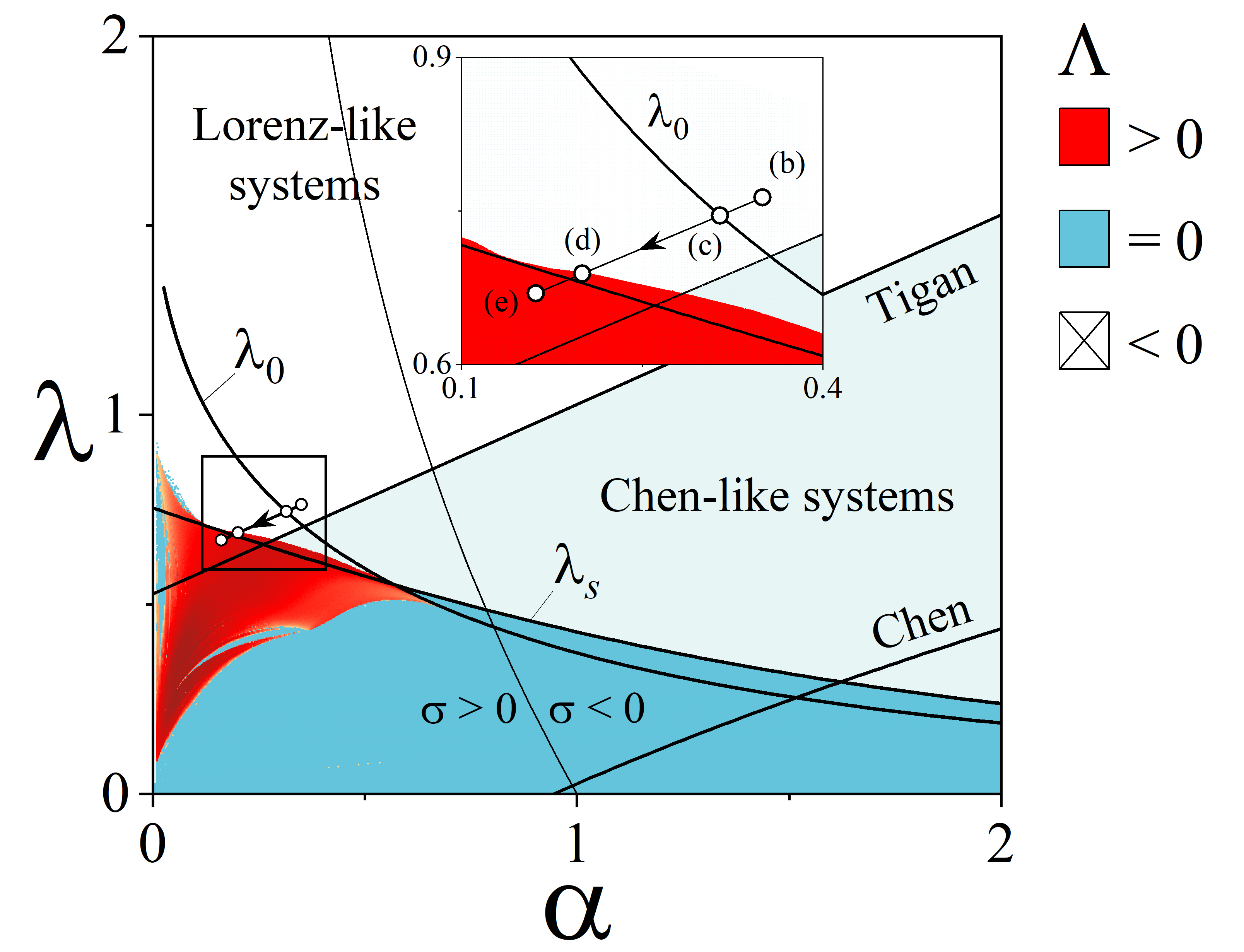}
(b)\includegraphics[width=0.43\linewidth]{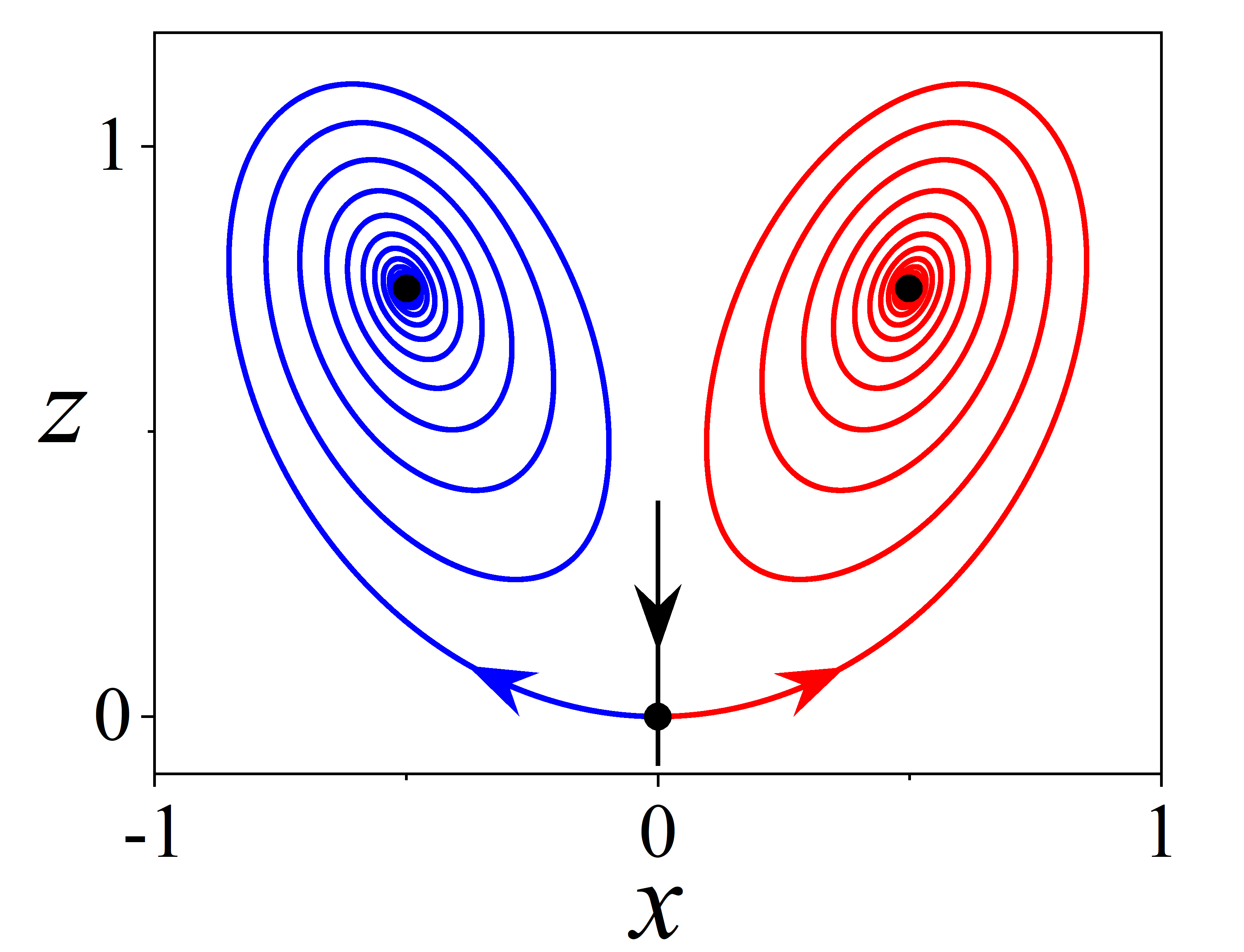}
(c)\includegraphics[width=0.43\linewidth]{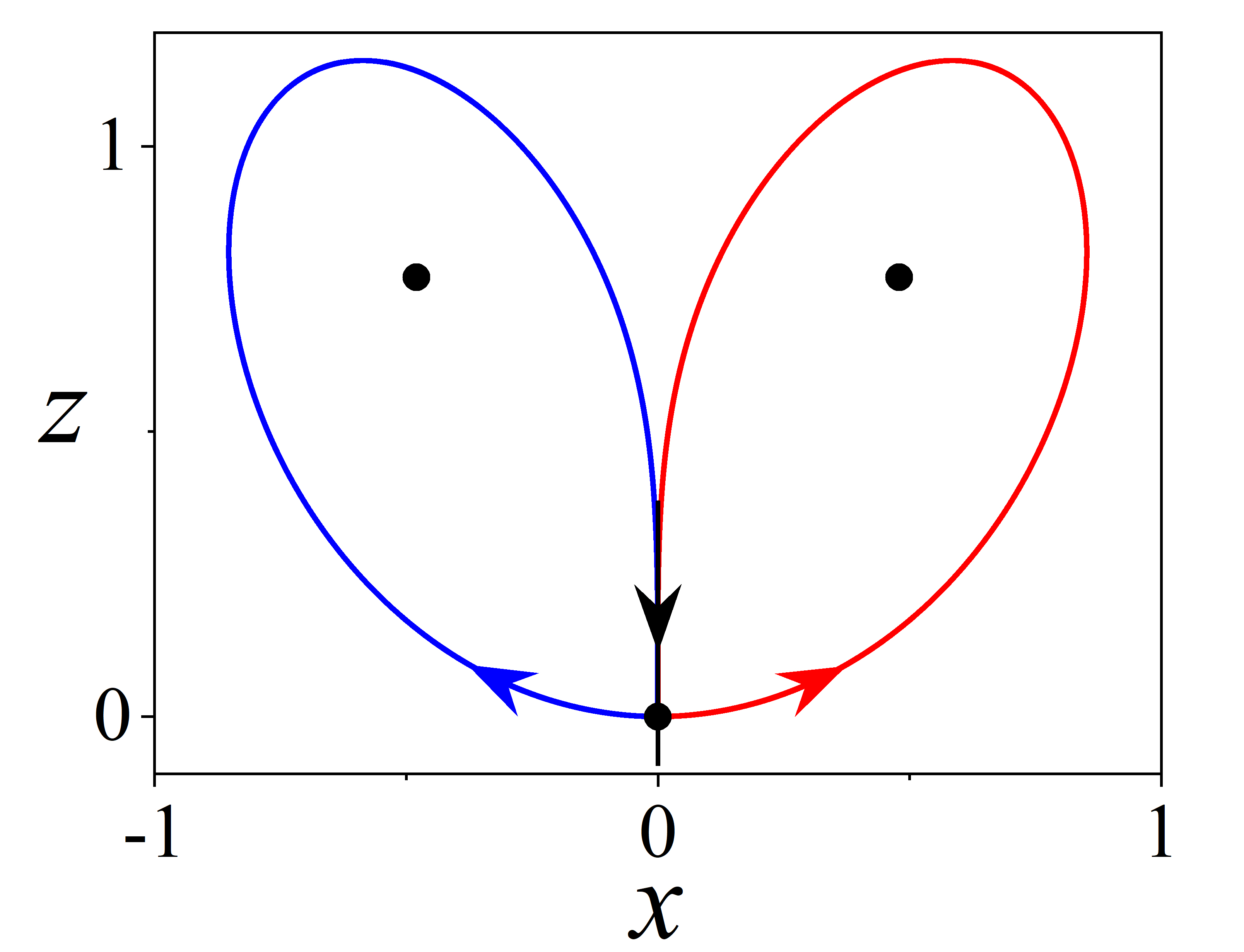}
(d)\includegraphics[width=0.43\linewidth]{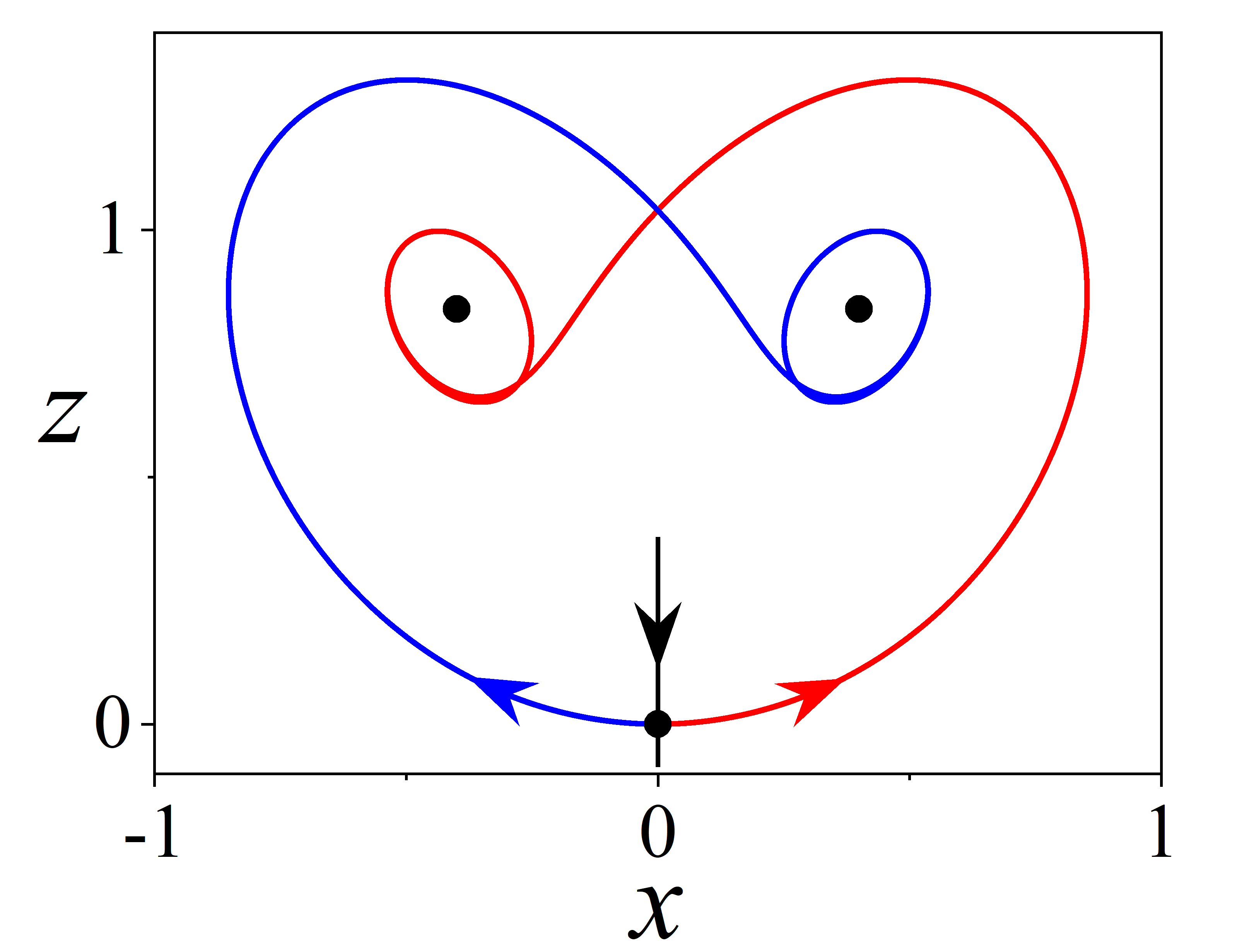}
(e)\includegraphics[width=0.43\linewidth]{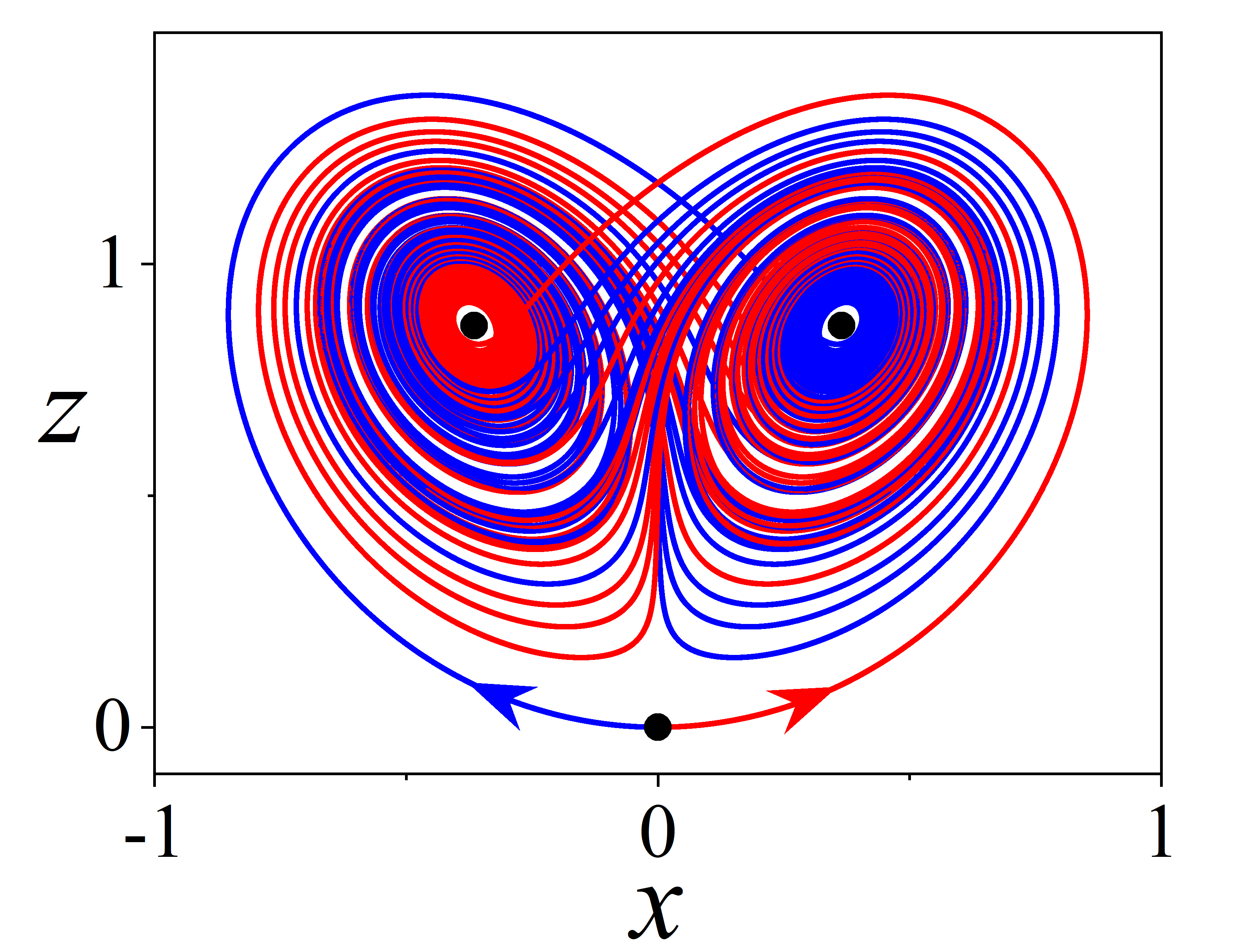}
\end{multicols}
    \caption{(a) Bifurcation diagram of UNF~\eqref{eq: main system} for $\beta=1.05487$, corresponding to original Lorenz attractor. As in Fig.~\ref{fig:UNF_systems_diag}, Tigan line  splits parameter plane into Lorenz (white) and Chen (gray-blue) regions. 
    $\Lambda$ is the largest Lyapunov exponent. Red denotes the parameter set of chaotic attractors ($\Lambda>0$), blue corresponds to existence of stable periodic orbits ($\Lambda=0$). Other regions correspond to trivial dynamics (stable foci $E^{\pm}$, $\Lambda<0$) and are not colored. The inset shows the bifurcation route (b)-(e) from trivial dynamics to chaos through principal homoclinic bifurcation $\lambda_{0}$, heteroclinic bifurcation (the boundary of white and red) and Andronov-Hopf bifurcation of foci $\lambda_{s}$ [see Eq.~\eqref{eq: two stable equilibria}].
    (b) $\alpha=0.35$, $\lambda=0.7634$: the foci are stable, the nontrivial set is absent. (c) $\alpha=0.314$, $\lambda=0.74564$: separatrices form the first homoclinic butterfly leading to the birth of nontrivial limiting set. (d) $\alpha=0.2005$, $\lambda=0.6865$: heteroclinic bifurcation leading to emergence of Lorenz strange attractor. (e) $\alpha=0.162$, $\lambda=0.6694$: Lorenz attractor is the unique attracting set.}
    \label{fig:UNF_bif_diag_Lorenz}
    \vspace*{12pt}
\end{figure*}

\begin{figure*}
 \centering
\begin{multicols}{2}
     (a)\includegraphics[width=1\linewidth]{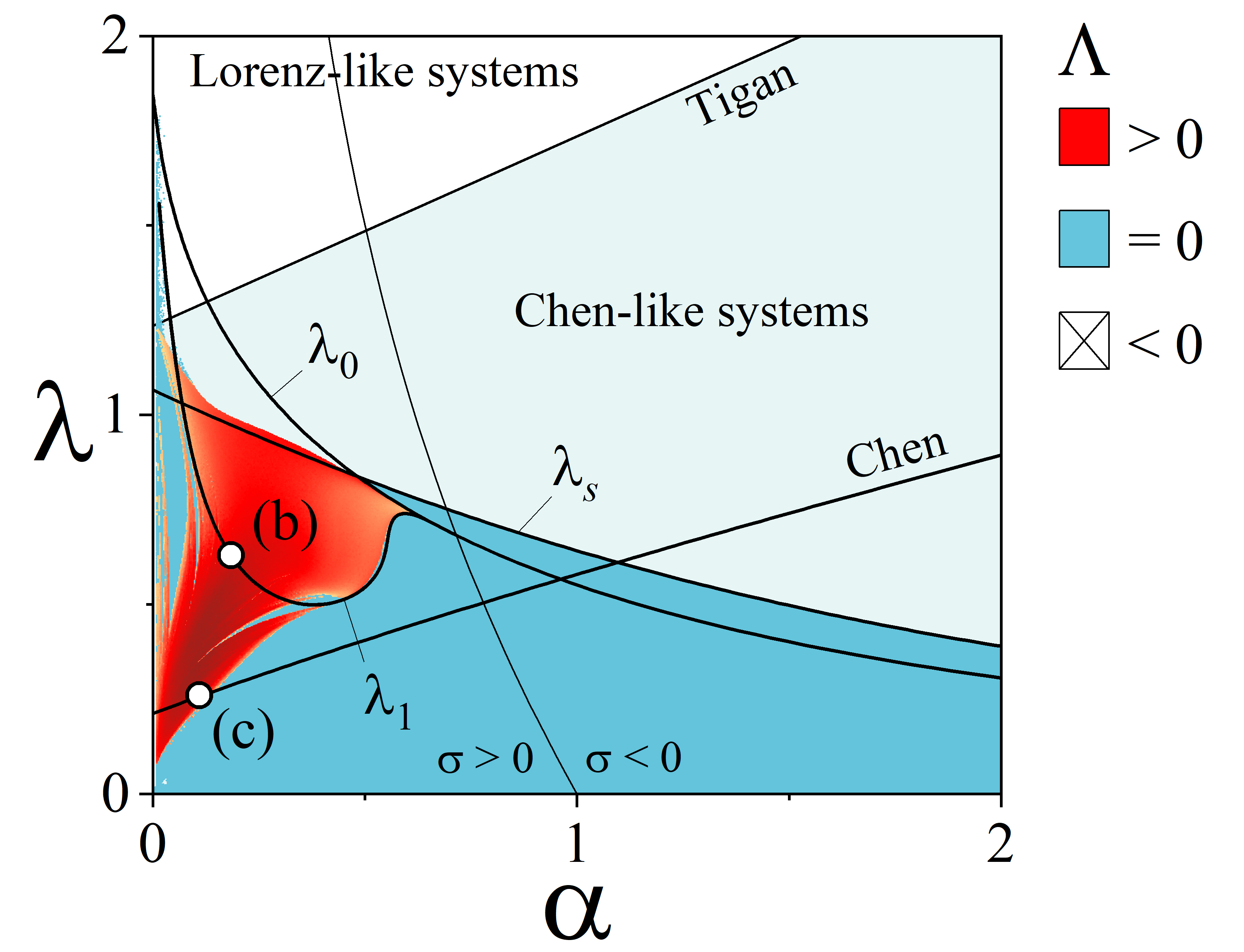}
    (b)\includegraphics[width=0.43\linewidth]{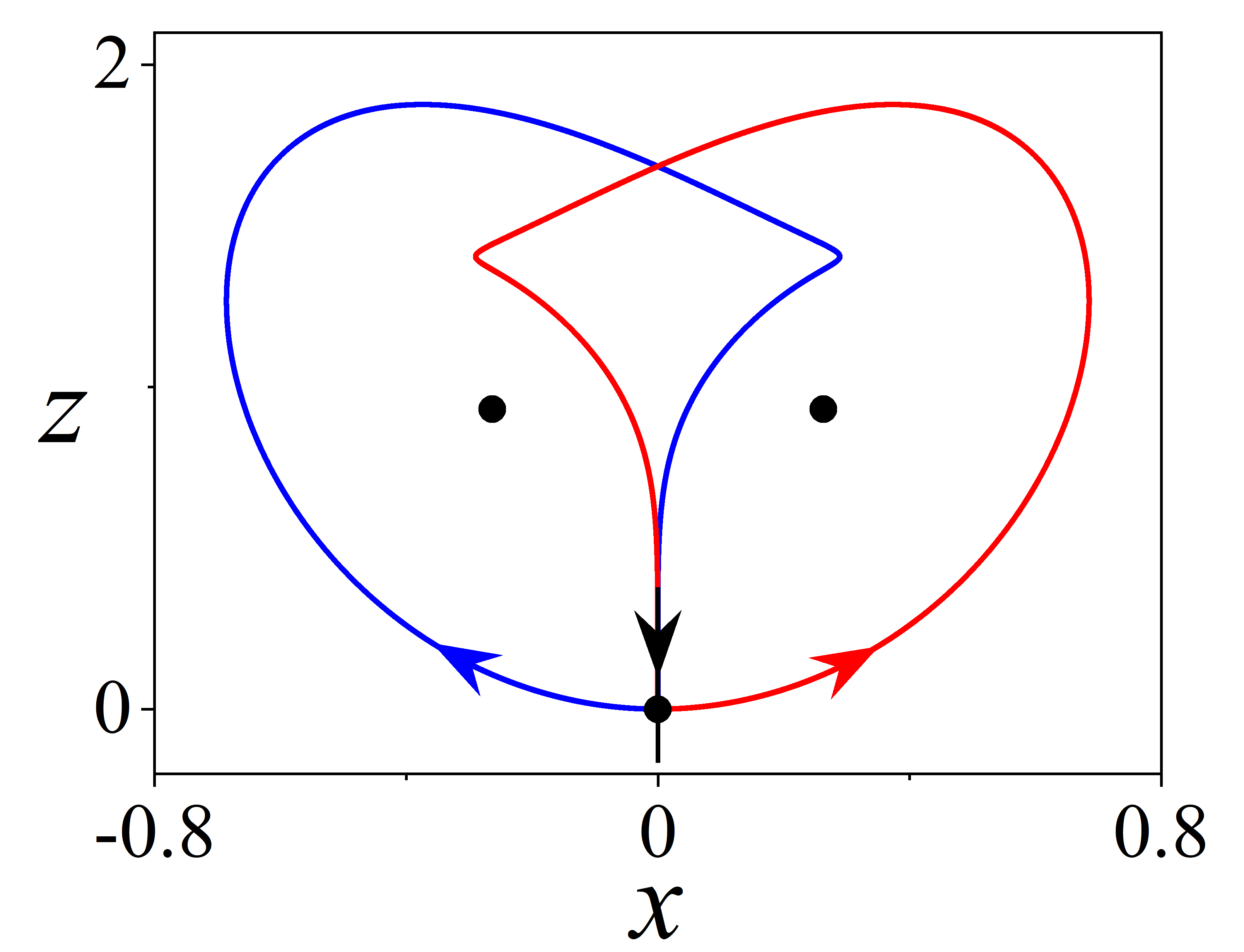}
    (c)\includegraphics[width=0.43\linewidth]{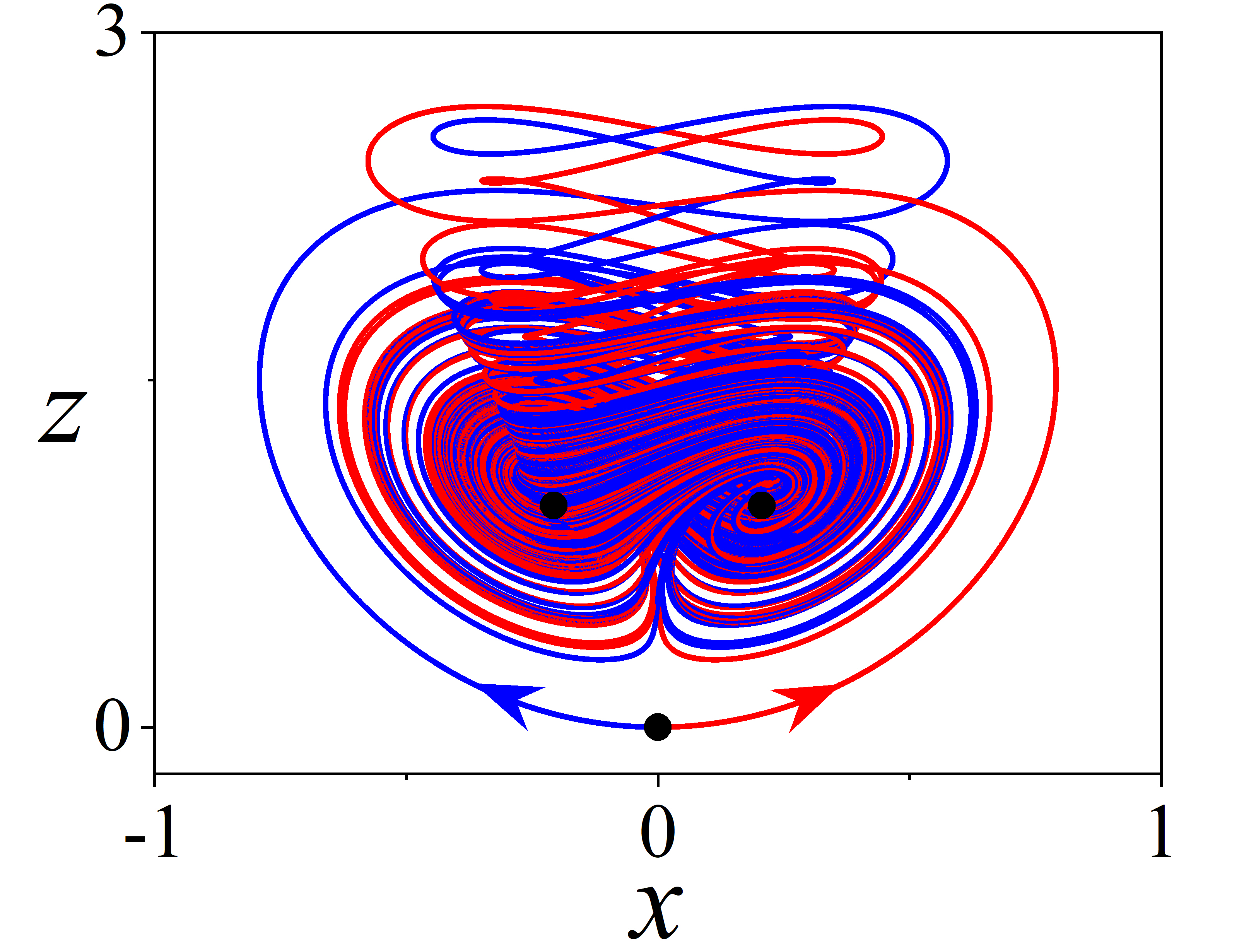}
    (d)\includegraphics[width=0.43\linewidth]{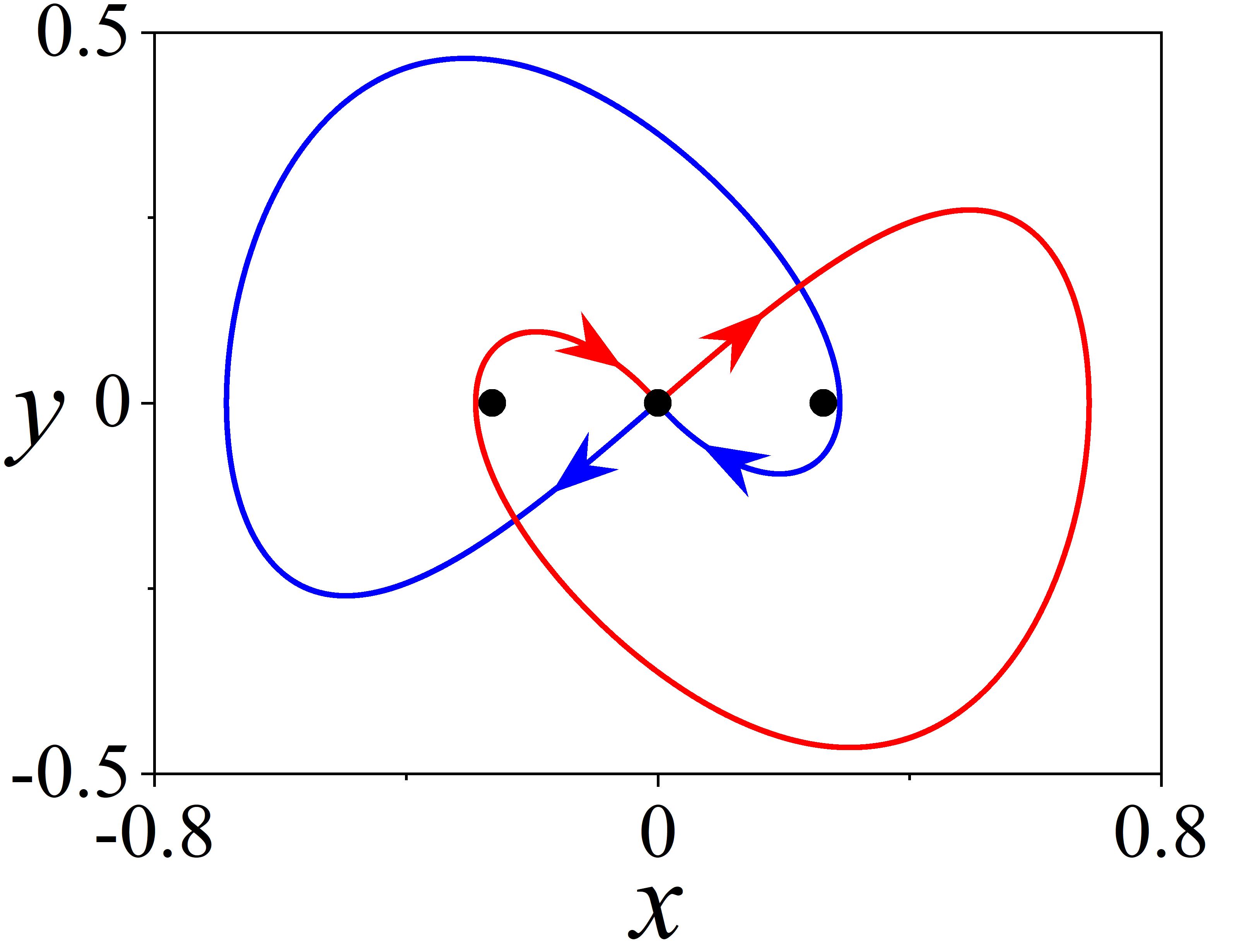}
    (e)\includegraphics[width=0.43\linewidth]{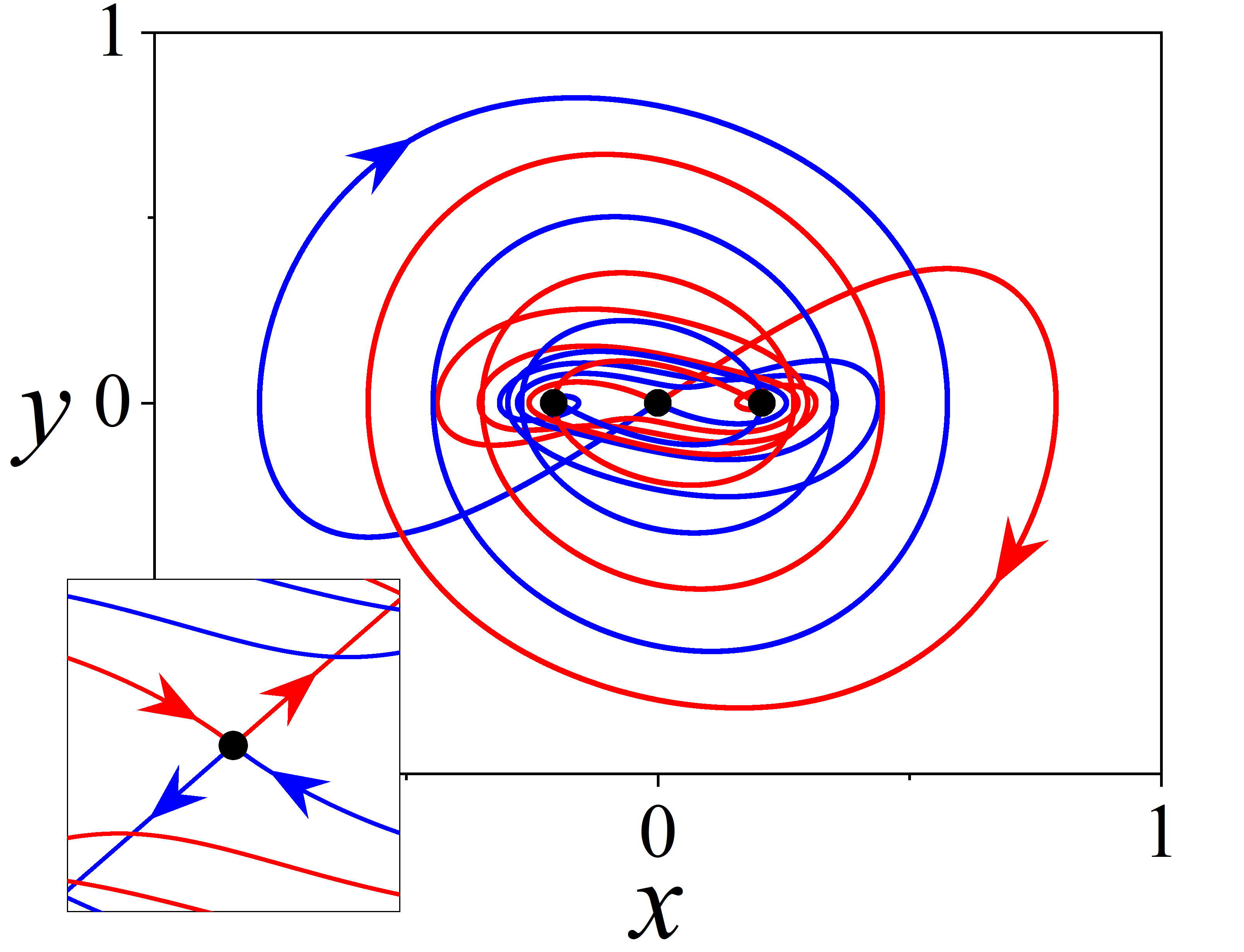}   
\end{multicols}
    \caption{(a) Bifurcation diagram of UNF~\eqref{eq: main system} for $\beta=2.47$ corresponding to original Chen attractor. All color coding and curves designations are the same as in Fig.~\ref{fig:UNF_bif_diag_Lorenz}. Curve $\lambda_{1}$ denotes the non-oriented homoclinic bifurcation. Phase pictures for circles (b) and (c) see at the panels (b), (d) and (c), (e), respectively.   
    Panels (b), (e): for $\alpha=0.184$, $\lambda=0.6296$, separatrices form the first non-oriented homoclinic butterfly. (c) $\alpha=0.11$, $\lambda=0.26$ the phase picture of original Chen attractor. (e) $xy$-projection of two symmetrical non-oriented twisted homoclinic orbits with multiple rotations around $z$-axis, existing in a small neighborhood of the parameters of the original Chen attractor. The inset contains an enlarged saddle vicinity.}
    \label{fig:UNF_bif_diag_Chen}
    \vspace*{12pt}
\end{figure*}

\section{Conclusions
\label{sec:Conclusions}}

We revealed that the system~\eqref{eq: main system} is related not only to original Lorenz system but serves the universal normal form (UNF) for dynamical systems with axial symmetry including first of all the Chen-like systems.
We improved the Theorem from \cite{belykh1984bifurcation} on existence of two symmetrical homoclinic orbits in Lorenz-like systems which indirectly generate the butterfly  structure of strange attractor. 
We proved the existence of infinite set of twisted homoclinic orbits that differ from each other in number of rotations around invariant axis of symmetry. 
Most of these bifurcations occur in the Chen zone of UNF parameters.
The original Chen attractor inherits the topology of twisted homoclinic orbit making finite number of such rotations.
For a small neighborhood of parameters $\lambda=0.26$, $\alpha=0.11$ and $\beta=2.47$ corresponding to the original Chen attractor, these homoclinic orbits are clearly observed [see Fig.~\eqref{fig:UNF_bif_diag_Chen}(e)].

Our analysis shows that decrease in parameter $\lambda$ causes an intricacy of the UNF bifurcation set.
Therefore dynamics in Chen zone of parameters is more complex than in Lorenz one.
In this regard, the ambitious paper \cite{algaba2013chen} looks ridiculous.

Between ``polar'' Lorenz and Chen systems one can meet numerous examples of regular, chaotic and quasi-strange attractors [for computer simulations for two fix values of the parameter $\beta$ see Fig.~\ref{fig:UNF_bif_diag_Lorenz}(a) and Fig.~\ref{fig:UNF_bif_diag_Chen}(a)].
Despite vast computational studies, the bifurcation structure in this broad parameter range remains far from a satisfactory analytical explanation.
All these problems are subject of a future study.

\section*{Acknowledgments}
\noindent This work was supported by the Russian Science Foundation under grant No. 24-21-00420. 

\bibliographystyle{ws-ijbc}
\bibliography{refs}

\end{multicols}
\end{document}